\documentclass[11pt]{article}
\usepackage[a4paper]{geometry}

\usepackage[utf8]{inputenc}
\usepackage[T1]{fontenc}
\usepackage{lmodern}
\usepackage[sectionbib,numbers]{natbib}

\usepackage{amstext,amsthm,amsmath,amssymb,mathtools,bbm,mathrsfs}
\usepackage{latexsym}
\usepackage[colorlinks=false]{hyperref}
\usepackage{graphicx}
\usepackage{enumerate}
\usepackage{graphicx}
\usepackage[normalem]{ulem}

\newtheorem{proposition}{Proposition}[section]
\newtheorem{corollary}[proposition]{Corollary}

\newtheorem{theorem}[proposition]{Theorem}

\newtheorem{lemma}[proposition]{Lemma}
\theoremstyle{definition}
\newtheorem{example}{Example}
\newtheorem{remarkx}[proposition]{Remark}

\newenvironment{remark}
  {\pushQED{\qed}\remarkx}
  {\popQED\endremarkx}

\newtheorem{xx}{\bf xxx}

\usepackage{color}

\usepackage[colorinlistoftodos,color=lightgray]{todonotes}

\usepackage{txfonts, upgreek}

\newcommand{\Z}{\mathbb{Z}}
\newcommand{\R}{\mathbb{R}}
\newcommand{\N}{\mathbb{N}}

\DeclareMathAlphabet{\mathpzc}{OT1}{pzc}{m}{it}

\newcommand{\ind}[1]{\mathbbm{1}_{\{#1\}}} 
\newcommand{\indset}[1]{\mathbbm{1}_{#1}} 
\providecommand{\abs}[1]{\lvert#1\rvert}

\providecommand{\norm}[1]{\lVert#1\rVert}

\newcommand{\wh}{\widehat}
\newcommand{\cadlag}{c\`adl\`ag{ }}

\DeclareMathOperator{\Span}{span}

\numberwithin{equation}{section}

\makeatletter
\newcommand{\subjclass}[2][2000]{%
  \let\@oldtitle\@title%
  \gdef\@title{\@oldtitle\footnotetext{#1 \emph{Mathematics subject classification:} #2}}%
}
\newcommand{\keywords}[1]{%
  \let\@@oldtitle\@title%
  \gdef\@title{\@@oldtitle\footnotetext{\emph{Key words and phrases:} #1.}}%
}
\makeatother

\begin{document}
\title{\LARGE Duality and the well-posedness of a martingale problem}

\author{\sc Andrej Depperschmidt, Andreas Greven, Peter Pfaffelhuber}

\date{\today}


\maketitle

\begin{abstract}
  \noindent
  For two Polish state spaces $E_X$ and $E_Y$, and an operator $G_X$, we obtain existence and uniqueness of a $G_X$-martingale problem provided there is a bounded continuous duality function $H$ on $E_X \times E_Y$ together with a dual process $Y$ on $E_Y$ which is the unique solution of a $G_Y$-martingale problem. For the corresponding solutions  $(X_t)_{t\ge 0}$ and $(Y_t)_{t\ge 0}$, duality with respect to a function $H$ in its simplest form means that the relation $\mathbb E_x[H(X_t,y)] = \mathbb E_y[H(x,Y_t)]$ holds for all $(x,y) \in E_X \times E_Y$ and $t\ge 0$. While duality is well-known to imply uniqueness of the $G_X$-martingale problem, we give here a set of  conditions under which duality also implies existence without using approximating sequences of processes of a different kind (e.g.\ jump processes to approximate diffusions) which is a widespread strategy for proving existence of solutions of martingale problems. Given the process $(Y_t)_{t\ge 0}$ and a duality function $H$, to prove existence of $(X_t)_{t\ge 0}$ one has to show that the r.h.s.\ of the duality relation defines for each $y$ a measure on $E_X$, i.e.\ there are transition kernels $(\mu_t)_{t\geq 0}$ from $E_X$ to $E_X$ such that $\mathbb E_y[H(x,Y_t)] = \int \mu_t(x,dx')\, H(x',y)$ for all $(x,y) \in E_X \times E_Y$ and all $t\geq 0$.

  As examples, we treat resampling and branching  models, such as the Fleming-Viot measure-valued diffusion and its spatial counterparts (with both, discrete and continuum space), as well as branching systems, such as Feller's branching diffusion. While our main result as well as all examples come with (locally) compact state spaces, we discuss the strategy to lift our results to genealogy-valued processes or historical processes, leading to non-compact (discrete and continuum) state spaces. Such applications will be tackled in forthcoming work based on the present article.
\end{abstract}

\section{Introduction}
A general method for constructing a class of time-homogeneous Markov processes on a Polish state space $E$ with measurable paths is by using \emph{martingale problems}, which we briefly recall.

Given a linear operator $G$ on a domain $\mathcal D$ which is a subspace of measurable, real-valued functions on $E$, and an initial law $\mathbb P_0 \in \mathcal M_1(E)$, the set of probability measures on $E$, we say that the distribution $\mathbb P$ of an $E$-valued \emph{progressively measurable} stochastic process $Z$ solves the martingale problem for $(G,\mathcal D,\mathbb P_0)$, if $\mathbb P(Z_0 \in \cdot\, ) = \mathbb P_0(\,\cdot\,)$ and
\begin{align}
  \label{eq:5}
  \Big(f(Z_t) - \int_0^t Gf(Z_s)\,ds\Big)_{t\geq 0}
\end{align}
is a $\mathbb P$-martingale (with respect to the filtration generated by $X$) for all $f \in \mathcal D$. By a martingale problem for $(G,\mathcal D,\,z)$ for a $z \in E$, we mean the martingale problem with initial measure $\mathbb P_0 = \delta_z$. The martingale problem for $(G,\mathcal{D},\mathbb P_0)$ is called \emph{well-posed}, if a solution exists and is unique. We say that the martingale problem for $(G,\mathcal{D})$ is well-posed if the martingale problem for $(G,\mathcal{D},\mathbb P_0)$ is well-posed for all $\mathbb P_0 \in \mathcal M_1(E)$.

\begin{remark}[Path regularity]
  \label{rem:R1}
  Recall that a solution of a martingale problem must have a modification with measurable paths to ensure existence of the integral in \eqref{eq:5} \cite[Section 4.3]{EK86}, and therefore has a progressively measurable modification \cite[Proposition 1.12]{KaratzasShreve1991}. So, without losing generality, the above definition uses a \emph{strong form of uniqueness} and a \emph{weak form of existence of solutions}, compared to formulations where it is required that the solutions have c\`adl\`ag paths, or where the initial states are restricted to be deterministic. This is  convenient because then we obtain a unique solution for which we have to prove \emph{regularity properties of paths separately}. For the latter, recall that on general state spaces, Theorem~4.3.6 in \cite{EK86} ensures the existence of a \cadlag modification of the solution of the $(G_X,\mathcal H_X)$ martingale problem provided the compact containment condition holds. 
\end{remark}

Duality, which we recall below, is a technique often used to show \emph{uniqueness} of solutions of a martingale problem. For \emph{existence} however, a typical strategy is to construct a tight sequence $Z^1,Z^2,\dots$ of approximating processes (typically some pure jump Markov processes), to prove tightness of the laws and to show that every limit point solves the martingale problem. The main goal of the paper is to use duality also for \emph{existence} of solutions of martingale problems; see Theorem~\ref{T1}. This approach avoids approximations with processes of a different nature than the solutions of the martingale problem. Note however that we also provide in Corollary~\ref{cor.480} a method to obtain solutions by approximations where the existence (and uniqueness) of the approximating sequences themselves is obtained using duality.

Two processes $X$ and $Y$ with Polish state spaces $E_X$ and $E_Y$, which arise as solutions of martingale problems $(G_X,\mathcal{D}_X)$ respectively $(G_Y,\mathcal{D}_Y)$, are said to be \emph{dual with respect to} a bounded, continuous function $H: E_X \times E_Y \to \R$, if
\begin{align}
  \label{eq:dual1}
  \mathbb E_{\mathbb P_0}[H(X_t,y)]
  = \int_{E_X} \mathbb E_y[H(x,Y_t)]\, \mathbb P_0(dx), \;
  \quad \mathbb P_0 \in \mathcal{M}_1(E_X), \; y \in E_Y,
\end{align}
where $\mathbb E_{\mathbb P_0}[\cdot]$ and $\mathbb E_y[\cdot]$ denote the expectations with respect to the initial conditions $X_0 \sim \mathbb P_0$ and $Y_0=y$, respectively. In particular, properties of $X$ can be read off from properties of $Y$ and vice versa. (We note that more general notions of duality exist, where one or both sides of \eqref{eq:dual1} contain an exponential penalty term, usually called Feynman-Kac-term; see \eqref{eq:toshow} below. Also, the boundedness of $H$ can be relaxed in which case some additional \emph{integrability conditions} have to be checked; see Remark~\ref{rem:genH}.)

Usually, \eqref{eq:dual1} is proved as follows (cf.\ (4.39)--(4.42) in Chapter~4 in \cite{EK86}): If $G_X$ and $G_Y$ are operators with domains $\mathcal D_X \supseteq \mathcal H_X \coloneqq \{H(\cdot,y):y\in E_Y\}$ and $\mathcal D_Y \supseteq \mathcal H_Y \coloneqq \{H(x,\cdot):x\in E_X\}$, respectively, and if $X$ and $Y$ are solutions of the corresponding martingale problems then \eqref{eq:dual1} is equivalent to
\begin{align}
  \label{eq:dual2}
  G_XH(\cdot,y)(x) = G_YH(x,\cdot)(y), \qquad x\in E_X, y\in E_Y.
\end{align}

In order to see that this suffices for \eqref{eq:dual1}, take a probability space where $X$ and $Y$ are independent and conclude from \eqref{eq:dual2} that
\begin{align}
  \label{eq:dual3}
  \frac{d}{ds} \mathbb E[H(X_s, Y_{t-s})] = \mathbb E[G_XH( \cdot , Y_{t-s})(X_s)] - \mathbb E[G_YH(X_s,  \cdot)(Y_{t-s})] = 0.
\end{align}
In addition, \eqref{eq:dual2} is necessary for \eqref{eq:dual1} since for $x\in E_X, y\in E_Y$
\begin{align*}
  G_XH(\cdot,y)(x) - G_YH(x,\cdot)(y) = \lim_{h\to 0} \frac 1h\left(\mathbb E_{\delta_x}[H(X_t, y)] - H(x,y) - \mathbb E_{\delta_y}[H(x, Y_t)] + H(x,y)\right) = 0.
\end{align*}

A classical result addresses the uniqueness of the martingale problem for $(G_X,\mathcal{D}_X,\mathbb P_0)$; see e.g.\ Proposition~4.4.7 and Remark 4.4.8 in \cite{EK86}. If $E_X$ and $E_Y$ are Polish, $\mathcal H_X$ is separating on the space of probability measures on $E_X$, and if for every $y\in E_Y$, there exists a solution $Y$ of the martingale problem for $(G_Y,\mathcal H_Y,y)$, and if \eqref{eq:dual1} holds for all $x\in E_X$ with $\mathbb P_0=\delta_x$ and $y\in E_Y$, then \emph{uniqueness} of the martingale problem for $(G_X,\mathcal H_X,x)$ holds. Also uniqueness of the more general martingale problems for $(G_X,\mathcal H_X,\mathbb P_0)$ with \emph{random} initial conditions $\mathbb P_0 \in \mathcal M_1(E_X)$ holds. The reason is that the duality relation \eqref{eq:dual1} and separability of $\mathcal H_X$ specify the one-dimensional distributions of $X$ uniquely, and therefore, by \cite[Theorem 4.4.2]{EK86}, uniqueness of the martingale problem follows.

Duality is also very useful if $Y$ is a much simpler process than $X$, because questions concerning the behaviour of $X$ can be translated to questions about $Y$. For example, duality can be used to show the \emph{Feller property of $X$}, or to determine its \emph{longtime behaviour} and characterize \emph{equilibria}. Duality was the key tool for studying interacting particle systems such as the voter model and the contact process \citep{Liggett:85}, but also for measure-valued processes such as the Fleming-Viot process (which is dual to some form of coalescent process; see also Examples~\ref{ex:FV}, \ref{ex:FVmut} and \ref{ex.1249}), and the Dawson-Watanabe superprocess (which is dual to the solution of a deterministic process given by a $\log$-Laplace equation) \cite{D93,Etheridge2000}. For a general reference on duality for Markov processes including various sorts of applications see \cite{EK86,JansenKurt2014} and references therein.

\medskip

The idea to use duality for the \emph{existence of a solution of a martingale problem} was motivated by constructions appearing in the literature. To the best of our knowledge, the first examples appear in \cite{EvansFleischmann1996, Evans1997}, where duality is used to show existence of the continuum space version of interacting Fisher-Wright diffusions on the discrete hierarchical group, by lifting the duality relation from the corresponding discrete case. This has been studied in on $\Z^1$ and $\R^1$ with other methods in \cite{KonnoShiga1988,Shiga1994,MullerTribe1995}. We believe that our approach provides proofs of theses results (when formulated differently) as well. The approach using duality is also used in \cite{GHKK14} to construct a spatial Cannings model, and in \cite{BEV10,EtheridgeVeberYu2020} for the construction of a model with locally constant population size in a spatial continuum. For a branching process, Dynkin gave in \cite{Dynkin1993} -- what he called -- a direct construction, which can be viewed as a construction based on the deterministic dual (as opposed to the construction via particle approximations in \cite{D93} for example).

We give here a systematic approach to the existence problem together with some examples. Let us briefly describe the idea for showing existence by using a dual process; see Theorem~\ref{T1} for all details: We are given the $(G_X,\mathcal{D}_X,\mathbb P_0)$ martingale problem for which we want to establish well-posedness. We look both for a Markov process $Y$ and a function $H$ for which the relation \eqref{eq:dual2} holds. Then we define the operator $P_t$ on $\mathcal H_X$ by setting $(P_tH(\cdot,y))(x) \coloneqq \mathbb E_y[H(x,Y_t)]$, which defines an operator on $\mathcal H_X$. Then $P_t$ inherits the semigroup property $P_t\circ P_s = P_{t+s}$ from the semigroup of the dual process $Y$. The semigroup $(P_t)_{t\geq 0}$ will be the semigroup of some process $X$, provided there is a probability measure $\mathbb P_x$ (with expectation $\mathbb E_x$) and for each $t \ge 0$ a random variable $X_t$ such that
\begin{align}\label{eq:a}
  \mathbb E_x[H(X_t,y)] = (P_tH(\cdot,y))(x) \, (\coloneqq \mathbb E_y[H(x,Y_t)]).
\end{align}
Then $(P_t)_{t\geq 0}$ is a Markov semigroup and we have \emph{existence} of a solution of the martingale problem for $(G_X,\mathcal H_X,x)$ provided some additional measurabiliy property holds. Moreover, if duality is derived from the operator criterion it also implies uniqueness. Altogether we obtain \emph{well-posedness} of the martingale problem for $(G_X,\mathcal H_X,x)$ for each $x \in E_X$. From that we obtain the well-posedness of the $(G_X,\mathcal H_X,\mathbb P_0)$ martingale problem for any $\mathbb P_0 \in \mathcal M_1(E_X)$. At least on compact state spaces, the existence of a \cadlag modification is immediate. 

The main requirement in applying our main result, namely Theorem~\ref{T1}, is to find (the distribution of) $X_t$ satisfying \eqref{eq:a}. For this, we provide two general approaches, one based on the Riesz-Markov Theorem in Proposition~\ref{l.480}, which requires \emph{compact} state spaces. In various applications, relaxing the assumption of compactness of $E_X$ is the main challenge. An approach in this direction is Proposition~\ref{l.480a} which requires $E_Y$ to be a set of functions on some compact set.

\medskip

For the construction of a solution of a martingale problem using duality we give several examples. Since our motivation came from \cite{Evans1997}, we also discuss here resampling systems with our approach. Namely in Examples~\ref{ex:FV}-- \ref{ex:slvf}, we show how our results can be used for the (spatial) Fleming-Viot process (with mutation) and the Cannings model, as well as the spatial $\Lambda$-Fleming-Viot process from \cite{Barton_2013}. In addition, we adapt arguments from \cite{Dynkin1993} and \cite{Beznea2011} in order to show existence in a continuous state branching model; see Example~\ref{ex:CSBP}. We also give an example how to use a Feynman-Kac term, by using the duality of the Feller branching diffusion to a Kingman coalescent; see Example~\ref{ex:Feller}. 

In future work, we want to systematize the approach to be able to construct \emph{genealogy-valued} processes based on martingale problems as introduced in \cite{GPWmp13,DGP12,GSW} and which could be generalized to genealogy-valued Fleming-Viot models with \emph{recombination} using arguments of the present paper to construct and characterize these new processes. Compare here also Section \ref{ss.outcomp} for more details. Another possibility is to use the approach to construct \emph{continuum space} dynamics, which was also the original motivation in \cite{EvansFleischmann1996} and this is taken up in work of Etheridge and coathors on $\R^d$ \cite{EtheridgeVeberYu2020} and in \cite{GSW} and subsequently on the continuum space hierarchical group extending \cite{GHKK14}.

\begin{remark}[Other methods for showing existence]
  Let us discuss two more options to show existence of solutions of a $(G_X,\mathcal D_X, \mathbb P_0)$ martingale problem without using a tight sequence of approximating processes: the positive maximum principle and the Girsanov transform. 
  \\For the former, consider locally compact $E_X$. Here, if (i) $G_X$ satisfies the \emph{positive maximum principle} (i.e.\ if $f\in\mathcal D_X$ and $x_0\in E$ such that $\sup_x f(x) = f(x_0)\geq 0$, then $G_Xf(x_0)\leq 0$) and (ii) is conservative (i.e.\ there is $f_1, f_2,\dots \in\mathcal D_X$ with $f_n \xrightarrow{n\to\infty} 1$ and $Gf_n \xrightarrow{n\to\infty} 0$ boundedly pointwise), existence follows (see e.g.\ \cite[Theorem~4.5.4 and Remark~4.5.5]{EK86}). However, we note that the positive maximum principle is very often not straight-forward to verify, for example in systems with infinitely many components.
  
  The \emph{Cameron-Martin-Girsanov theorem} is another way to show existence of solutions of the $(G_X, \mathcal D_X, \mathbb P_0)$ martingale problem for, given (i) existence of a process $Z$, (ii) a mean-1-martingale $M\geq 0$ and (iii) a proof that $(M\cdot \mathbb P)_\ast Z$ (here $M\cdot \mathbb P$ denotes the probability measure with density $M$ with respect to $\mathbb P$) solves the $(G_X,\mathcal D_X, \mathbb P_0)$ martingale problem. However, it might here be necessary to prove existence of the process $Z$ by some other methods, for instance by again using approximation techniques or the positive maximum principle.
\end{remark}

For future reference we introduce in the following remark the notation used throughout the paper. The reader might skip it and return to it if the notation that we use is not familiar.
\begin{remark}[Notation and some basic concepts]
  Throughout, let $(E,r)$ be a complete and separable metric space. Also, let $\mathcal C_b(E)$ and $\mathcal B(E)$ be the spaces of real-valued, continuous and bounded respectively bounded measurable functions.  With a slight abuse of notation, we also write $\mathcal B(E)$ for the set of Borel-measurable subsets of $E$. On $\mathcal C_b(E)$, we use the supremum norm $\norm{\cdot}$ and equip $\mathcal C_b(E)$ with the bounded pointwise (bp)-topology where $f_n \to f$ iff $\sup_n \norm{f_n} < \infty$ and $f_n\to f$ pointwise. We denote by $\mathcal M(E)$ ($\mathcal M_1(E)$) the space of (probability) Radon measures on $E$. If $E$ is locally compact, we denote by $\widehat C(E) \subseteq \mathcal C_b(E)$ the set of continuous functions vanishing at infinity. For $E$-valued random variables $Y,Z$, we write $Y\sim Z$ or $Y \sim \mathcal L(Z)$ if $Y$ and $Z$ have the same distribution.

  We say that $\Pi \subseteq \mathcal C_b(E)$ is separating (on $\mathcal M_1(E)$) if for all $\mu, \nu \in\mathcal M_1(E)$, ($\int f d\mu = \int f d\nu$ for all $f\in \Pi)\Rightarrow \mu = \nu$ holds, and convergence determining (in $\mathcal M_1(E)$) if, for all $\mu, \mu_1,\mu_2,\dots$, ($\int f d\mu_n \xrightarrow{n\to\infty} \int f d\mu$ for all  $f\in\Pi$) $\Rightarrow (\mu_n \xRightarrow{n\to\infty} \mu$) holds.

  Recall that a semigroup $(P_t)_{t\geq 0}$ on a vector space $\mathcal D \subseteq \mathcal C_b(E)$ is a family of bounded linear functions $P_t: \mathcal D \to \mathcal B(E)$ such that $P_t(P_s f) = P_{t+s}f$ for all $t,s\geq 0$ and $f\in \mathcal D$ with $P_sf\in \mathcal D$. The operator $P_t$ (or the semigroup $(P_t)_{t\geq 0}$) is a contraction if $\norm{P_tf} \leq \norm{f}$ (for all $t\geq 0$). It is positive if $P_tf\geq 0$ for $f\geq 0$. It is conservative if $P_t1 = 1$. A semigroup $(P_t)_{t\geq 0}$ is called strongly continuous if $P_tf \xrightarrow{t\to 0}f$ for all $f\in\mathcal C_b(E)$. If a conservative, positive, strongly continuous contraction semigroup satisfies $P_tf \in \mathcal C_b(E)$ for $f\in\mathcal C_b(E)$ and $t\geq 0$, we call $(P_t)_{t\geq 0}$ a $\mathcal C_b(E)$-Feller semigroup. If the same holds for locally compact $E$ with $\widehat{\mathcal C}(E)$ instead of $\mathcal C_b(E)$, then we say that $(P_t)_{t\geq 0}$ is a $\widehat{\mathcal C}(E)$-Feller semigroup. The generator of a semigroup $(P_t)_{t\geq 0}$ is given by $Gf(x) = \lim_{t\to 0} \tfrac 1t (P_tf(x) -f(x))$, whenever the limit exists boundedly pointwise. The set $\mathcal D(G)$ of functions for which the limit exists boundedly pointwise is referred to as the domain of the generator $G$.

  Recall that with any time-homogeneous Markov process $X=(X_t)_{t\ge 0}$ on a state space $E$ we can associate a semigroup $P=(P_t)_{t\geq 0}$ with $P_tf(x) = \mathbb E_x[f(X_t)]$ satisfying the Chapman-Kolmogorov  equations $P_tP_sf = P_{t+s}f$ for $s,t\geq 0$. This semigroup is a positive, conservative contraction. We say that $X$ is a Feller process if its semigroup is Feller (with respect to either $\mathcal C_b(E)$ or $\widehat{\mathcal C}(E)$).
\end{remark}

\section{Results}
We will first present in Theorem~\ref{T1} the general result on the well-posedness of a martingale problem using duality in Section~\ref{ss.princres}. Then in Section~\ref{ss.checkT1}, we will discuss how to check the assumptions appearing in Theorem~\ref{T1}.  In Section~\ref{ss.combine}, we show how our results can be applied to processes whose generators consist of sums of generator terms each of which corresponds to different mechanisms of the process and which we can characterize by a martingale problem for which we have a duality. Proofs or arguments for results are found in Section~\ref{s.proofresults}. Several examples are treated in Section~\ref{s.examples}. Finally, in the Outlook-Section~\ref{ss.outcomp} we discuss how the restrictions of our results to compact state spaces can be used for non-compact and in particular non locally compact cases by checking additional conditions.

\subsection{The principal result}
\label{ss.princres}
Theorem~\ref{T1} below is our main result for showing existence of solutions of martingale problems. We will say that two processes $X$ and $Y$ (with state spaces $E_X$ and $E_Y$ are in $H$-duality (for some $H: E_X\times E_Y \to\mathbb R$) with potential $\beta: E_Y\to\mathbb R$ if
\begin{align}
    \label{eq:dualHbeta}
    \mathbb E_x[H(X_t, y)] = \mathbb E_y\Big[H(x,Y_t)\exp\Big( - \int_0^t \beta(Y_r)dr\Big) \Big].
\end{align}
Note that -- in contrast to the introduction -- we are dealing with the slightly more complex situation because \eqref{eq:dualHbeta} involves (in contrast to \eqref{eq:dual1}) an extra term on the right-hand-side, often referred to as a \emph{Feynman-Kac term)}, denoted here by $\beta$. In various applications which we present in Section~\ref{s.examples}, we will have $\beta=0$; in Example~\ref{ex:Feller} we treat a case for $\beta \neq 0$. The proof of the following result is given in  Section~\ref{s.proofresults}. 

\begin{theorem}[A semigroup property and existence by duality]\label{T1}
  Let $E_X, E_Y$ be Polish, $H: E_X\times E_Y \to \R$ bounded and continuous, and $G_Y : \mathcal H_Y \to \mathcal C_b(E_Y), \beta \in \mathcal C_b(E_Y)$. Define $\mathcal H_X \coloneqq \{H(\cdot,y): y\in E_Y\}$ and $\mathcal H_Y \coloneqq \{H(x,\cdot): x\in E_X\}$. 
  \begin{enumerate}[(i)]
      \item \label{th1i} Suppose that for each $y\in E_Y$ there is an $E_Y$-valued Markov process $Y$ with a strongly continuous semigroup, which is the unique solution of the $(G_Y,\mathcal{H}_Y, y)$-martingale problem. Then, the family $(P_t)_{t\geq 0}$, defined on the closure of $\text{span}(\mathcal H_X)$, given by
      \begin{align}\label{eq:t1a}
          P_t H(.,y) := \mathbb E_{y}\Big[ H(., Y_{t}) \exp \Big( - \int_0^{t} \beta(Y_r) dr\Big) \Big],
      \end{align}
      is a semigroup. Assume that its generator $G_X$ has domain $\mathcal D_X \supseteq \mathcal H_X$ and  satisfies 
      \begin{align}
        \label{eq:dual3new}
        G_XH(\cdot,y)(x) = G_Y H(x,\cdot)(y) + \beta(y) H(x,y), \quad
        x\in E_X,\, y\in E_Y.
      \end{align}    
    \item \label{th1ii} In addition, assume that $\Span(\mathcal H_X)$ is separating on $\mathcal M_1(E_X)$ and there exists a family $(\mu_t)_{t\geq 0}$ of probability kernels from $E_X$ to $E_X$ such that, for all $\Gamma \in \mathcal B(E_X)$,
    \begin{align}
      \label{eq:toshow0}
      (t,x)\mapsto \mu_t(x,\Gamma) \text{ is $\mathcal B([0,\infty)
      \times E_X) -\mathcal B([0,1]) $ measurable, }
    \end{align}
    and for all $y\in E_Y$ and $t\geq 0$ the \emph{kernel representability} condition
    \begin{align}
      \label{eq:toshow}
      P_t H(.,y) = \int_{E_X} \mu_t(.,dx') H(x',y)
    \end{align} 
    holds. Then, for each $x\in E_X$, there exists a Markov process $X = (X_t)_{t\geq 0}$ starting in $x$ and having transition kernels $(\mu_t)_{t\geq 0}$, i.e.\ the right hand side of  \eqref{eq:toshow} equals $\mathbb E_x[H(X_t,y)]$. In particular, $X$ and $Y$ are in duality w.r.t.\ $H$ and potential $\beta$. Moreover, the process $X$ is the unique solution of the martingale problem for $(G_X,{\mathcal H}_X,x)$ and the martingale problem for $(G_X,{\mathcal H}_X)$ is well-posed.    
  \item Finally, if $\Span(\mathcal{H}_X)$ is convergence determining, then $X$ is $\mathcal C_b(E_X)$-Feller.
  \end{enumerate}
\end{theorem}

We note that Theorem~\ref{T1} is concerned with martingale problems for $(G_X,{\mathcal H}_X)$ and does not make any statements about existence of solutions of the martingale problem for $(G_X,\mathcal D_X)$ at this point. The uniqueness of the solution is of course immediate. The  step from $\mathcal{H}_X$ to $\mathcal{D}_X$ is an application of general theory; see Section~4.3 in \cite{EK86}. Using Proposition~4.3.1 of~\cite{EK86} we obtain the following corollary to Theorem~\ref{T1}.

\begin{corollary}[Well-posedness of martingale problems]
  Assume that the
  \begin{align}
    \label{eq:38}
    \text{bp-closures of } \;
    \{(f, G_Xf): f\in \mathcal{H}_X\} \; \text{ and of } \;
    \{(f, G_Xf): f\in \mathcal D_X\} \; \text{ agree},
  \end{align}
  and that the assumptions of Theorem~\ref{T1} are satisfied. Then the martingale problem for $(G_X,\mathcal D_X)$ is
  well-posed.
\end{corollary}

In the case of a locally compact state space $E_X$, recall from Theorem~4.2.7 in \cite{EK86} that $\wh{\mathcal C}(E_X)$-Feller semigroups generate strong Markov processes with \cadlag paths. We give the corresponding result in our case only for compact state spaces, since Theorem~\ref{T1}(ii) in general only gives the $\mathcal C_b(E_X)$-Feller property.

\begin{corollary}[Path regularity]
  \label{p.loccomp-cadl}
  Let $E_X$ be compact and let the assumptions of Theorem~\ref{T1} be satisfied and assume that $X$ is the process obtained in Theorem~\ref{T1}(a). If $\Span(\mathcal{H}_X)$ is convergence determining and $Y$ is $\mathcal C_b(E_Y)$-Feller, then $X$ has a modification with \cadlag paths.
\end{corollary}

\begin{proof}
  By Theorem~\ref{T1}(ii), $X$ is ${\mathcal C}(E_X)$-Feller, since $E_X$ is compact. Then, the result follows from Corollary~4.3.7 in \cite{EK86}.
\end{proof}

\begin{remark}[More general choices of $H$]
  \label{rem:genH}
  In Theorem~\ref{T1} we have assumed that $H$ is a bounded and continuous function. By inspection of its proof one can see that the boundedness assumption is only used in a calculation that uses a Fubini argument. Thus, if $H$ is nonnegative or if the condition  
  \begin{align}
    \label{eq:toshow-inf}
    \mathbb E_y\Big[\abs{H(x,Y_t)} \exp\Big(\int_0^t \beta(Y_s)\,
    ds\Big)\Big] <\infty,  \quad y \in E_Y, \; t \ge 0,
  \end{align}
  is fulfilled, then the assertions of Theorem~\ref{T1} remain true for unbounded $H$.
\end{remark}

\subsection{Checking the conditions of Theorem~\ref{T1}}
\label{ss.checkT1}
In this subsection we discuss \emph{how to check the conditions} of Theorem~\ref{T1}. First, let us note that in Theorem~\ref{T1}, it is not a restriction to assume that $Y$ is Markovian because the Markov property of $Y$ follows from the well-posedness of the martingale problem $(G_Y,\mathcal{H}_Y)$; see e.g.\ Theorem~4.4.2 in \cite{EK86}. Thus, if we have a process $Y$ and a function $H$ satisfying the generator relation \eqref{eq:dual3new}, it remains to check the assumptions \eqref{eq:toshow0} and \eqref{eq:toshow}.

In the following proposition we provide sufficient conditions for the measurability assumption \eqref{eq:toshow0}. The proof can be found in Section~\ref{s.proofresults}.

\begin{proposition}[Sufficient conditions for \eqref{eq:toshow0}]
  \label{l3}
  Let $H$, $\mathcal{H}_X$, $Y=(Y_t)_{t\geq 0}$ Markov with a strongly continuous semigroup be as in Theorem~\ref{T1}, and let $(\mu_t)_{t\geq 0}$ be as in \eqref{eq:toshow}. If $\Span(\mathcal{H}_X)$ is convergence determining, then the following assertions hold:
  \begin{enumerate}[(i)]
  \item The mapping
    $(t,x)\mapsto \mu_t(x,\cdot) \in \mathcal M_1(E_X)$ is continuous.
  \item For all $x\in E_X$, $\Gamma \in \mathcal B(E_X)$, the mapping
    $(t,x) \mapsto \mu_t(x,\Gamma)$ is measurable, i.e.\
    \eqref{eq:toshow0} holds.
  \end{enumerate}
\end{proposition}

\noindent
The key condition that remains to be checked is \eqref{eq:toshow}.
For this, recall the following version of the \emph{Riesz-Markov theorem}: If $E$ is compact and $P: \mathcal C(E) \to \mathcal C(E)$ is linear, positive (i.e.\ $f\geq 0$ implies $Pf\geq 0$) and $P1 = 1$, there is $\mu \in\mathcal P(E)$ such that $Pf = \int f d\mu$. We will use this theorem in two ways (always for fixed $x\in E_X$ and $t\geq 0$, and denoting the left hand side of \eqref{eq:toshow} by $P_t^xy$).

\emph{First} (see Proposition~\ref{l.480}), set $E = E_X$ and assume that span$(\mathcal H_X)$ is dense in $\mathcal C(E_X)$. Then, if the map on span$(\mathcal H_X)$, given by $H(,.y)\mapsto P_t^xy$, is positive, we can extend this to $\mathcal C(E_X)$ and find the corresponding $\mu_t(x,.)$.

\emph{Second}, (see Proposition~\ref{l.480a}), assume that $E_Y$ is a vector space which is dense in $\mathcal C(E_U)$ for some compact $E_U$, and assume that $y\mapsto P_t^xy$ is a positive linear form. Then, we find some $U$-valued random variable (due to the Riesz-Markov theorem) with $P_t^x y = \mathbb E[y(U)]$. Now, if it is possible to find $X$ with $\mathbb E[H(X,y)] = \mathbb E[y(U)]$, we can take $\mu_t(x,.)$ as the distribution of $X$.

We note that many measure-valued processes of interest lead to compact (or locally compact) state spaces. However, recall that the \emph{historical processes} (see \cite{DP91}) for branching models and for Fleming-Viot processes have state spaces which are \emph{not} locally compact. The same is typically true for genealogy-valued processes; see Remark~\ref{ss.outcomp} and \cite{GPWmp13,DGP12,GSW}. Strategies how to use Theorem~\ref{T1} in such situations are discussed in Section~\ref{ss.outcomp}.

\medskip
\noindent
Now, we give two conditions which can be used to show
\eqref{eq:toshow}.

\begin{proposition}[A way to check condition \eqref{eq:toshow} for compact $E_X$]
  \label{l.480}
  Let $E_X$, $E_Y$, $H$, $\beta$, $Y$ and $P_t$ be as in Theorem~\ref{T1}(i). Assume that
  \begin{enumerate}[(i)]
  \item $E_X$ is compact;
  \item $\Span(\mathcal{H}_X) \subseteq {\mathcal C}(E_X)$ is a convergence determining algebra (i.e.\ it is closed under multiplication) containing~1;
  \item the semigroup of $Y$ is $\mathcal C_b(E_Y)$-Feller;
  \item for all $t\geq 0$ and $x\in E_X$, the linear map  $P_t^x: \Span(\mathcal H_X) \to \mathbb R$, given by
    \begin{align}
      \label{eq:16}
      P_t^x H(\cdot,y) \coloneqq P_tH(\cdot,y)(x) = \mathbb E_{y}\Bigl[H(x,Y_t) \exp\Bigl(\int_0^t     \beta(Y_s)\, ds\Bigr)\Bigr]
    \end{align}
    is positive with $P_t^x 1 = 1$.
  \end{enumerate}
  Then, there is a unique continuous extension of $(P_t)_{t\geq 0}$ to ${\mathcal C}(E_X)$, which is again a positive linear form. Moreover, there is a family of probability kernels $(\mu_t)_{t\geq 0}$ from $E_X$ to $E_X$ such that \eqref{eq:toshow} holds.
\end{proposition}

The verification of the point (iv) is based on using the properties of $H(\cdot,y)$ and the form of the states of the dual process $Y$. In the examples from Section~\ref{s.examples}, we will e.g.\ apply moment problems for checking (iv). In general, there are more applications of the proposition than one might think as we will see in the section on examples.

\medskip

In some cases, verification of (iv) is possible by using approximate dual processes, as in the following corollary to Proposition~\ref{l.480}. It will be used in Example~\ref{ex:slvf}. Again the proof can be found in Section~\ref{s.proofresults}.

\begin{corollary}[Approximating duals for (iv) of Proposition~\ref{l.480}]\label{cor.480}
    Let $E_X, E_Y, H, \beta, Y$ and $P_t$ be as in Theorem~\ref{T1}(i) and assume that (i)--(iii) of Proposition~\ref{l.480} hold. 
    \\
    In addition, let $X^1, X^2,\dots$ be Markov processes (with state space $E_X$), and $Y, Y^1, Y^2,\dots$ be Markov processes (with state space $E_Y$) such that $X^n, Y^n$ are in $H$-duality with potential $\beta$ (see \eqref{eq:dualHbeta}), $n=1,2,\dots$, and $Y^n\xRightarrow{n\to\infty} Y$. Then, (iv) of Proposition~\ref{l.480} holds. In particular, \eqref{eq:toshow} of Theorem~\ref{T1} holds. 
\end{corollary}

\medskip

In many situations it is necessary to work with \emph{function-valued duals}, for instance in population genetics, if we deal with measure-valued processes (measures on some type space $I$ which is often compact) and if the mechanisms include mutation and selection. Then duality functions are functions on the space of the samples on $\mathcal U = I^\N$ for some compact type space $I$ and hence $\mathcal U$ is compact. We now give a second condition for verifying \eqref{eq:toshow}, which applies in the situation where $E_Y$ is a \emph{space of continuous functions} on some compact set.

\begin{proposition}[Another way to check condition \eqref{eq:toshow}]
  \label{l.480a}
  Let $E_X$, $E_Y$, $H$, $\beta$, $Y$ and $P_t$ be as in Theorem~\ref{T1}(i) with $E_Y$ being a set of continuous functions to be specified below. Assume that
  \begin{enumerate}[(i)]
  \item the semigroup of $Y$ is $\mathcal C_b(E_Y)$-Feller;
  \item there exists a compact metric space $E_U$ so that $E_Y\subseteq {\mathcal C}(E_U)$ is a vector space containing~1, which is dense (with respect to the $\sup$-norm) in ${\mathcal C}(E_U)$;
  \item for all $t\geq 0$ and $x\in E_X$, the linear map $Q_t^x: \text{span}(E_Y) \to \mathbb R$, given by
  \begin{align}
    \label{eq:8}
    Q_t^x y & \coloneqq P_t H(.,y)(x) = \mathbb E_{y}\Bigl[H(x,Y_t)
              \exp\Bigl(\int_0^t \beta(Y_s)\, ds\Bigr)\Bigr]
  \end{align}
  is a positive linear form with $Q_t^x 1 = 1$;
  \item for any $E_U$-valued random variable $U$ (with $E_U$ from (ii)), there is an $E_X$-valued random variable $X$ such that
    \begin{align}
      \label{eq:abc}
      \mathbb E[H(X,y)] = \mathbb E[y(U)]\, \text{ for all }\, y \in E_Y.
    \end{align}
  \end{enumerate}
  Then, there is a family of probability kernels $(\mu_t)_{t\geq 0}$
  from $E_X$ to $E_X$ such that \eqref{eq:toshow} holds.
\end{proposition}
In Example~\ref{ex:FVmut}, we will apply Proposition~\ref{l.480a} via the following corollary which contains easier to check conditions. Recall measure-valued processes on $\mathcal M(I)$ and that duality processes in this case are function-valued with functions depending on samples from $\mathcal U = I^\N$. 

\begin{corollary}[How to check \eqref{eq:abc}]
  Let \label{cor:htc_abc} $Y^y$ denote the stochastic process, distributed according to $Y$ with initial value $y$. In the situation of Proposition~\ref{l.480a}, we can replace (iv) by one of the following conditions.
  \begin{enumerate}
  \item[(iv')] Suppose $\beta=0$ and there is a subset $\mathcal F \subseteq \{f: E_U \to E_U \text{ measurable}\}$ with $E_Y \circ \mathcal F\coloneqq \{y\circ f: y\in E_Y, f\in \mathcal F\} \subseteq E_Y$ and   furthermore that $Y$ and $\mathcal F$ are such that for all $x\in E_X$ and $f\in\mathcal F$ (with $G_Y$ the generator of $Y$)
  $H(x,y \circ f) = H(x,y)$ and
  \begin{align}\label{eq:ycircf}
      G_YH(x,\cdot)(y) = G_YH(x,\cdot)(y\circ f).
  \end{align}
  Then, for any $E_U$-value random variable $U$ with $U \sim f(U)$ for all $f\in\mathcal F$, there is an $E_X$-valued random variable $X$ such that \eqref{eq:abc} holds. 
  \item[(iv'')] If $O\subseteq E_U$ is such that there are $y_1, y_2,\ldots \in E_Y$ such that $y_{n} \xrightarrow{n\to\infty} \indset{O^c}$  boundedly pointwise and $H(x,Y_t^{y_{n}})\xrightarrow{n\to\infty}0$ in probability for all $t\geq 0$. Then, for any $O$-valued random variable $U$, there is an $E_X$-valued random variable $X$ such that \eqref{eq:abc} holds. 
  \end{enumerate}
\end{corollary}

\begin{remark}[Using (iv') and (iv'')]
  We will use condition (iv') in Example~\ref{ex:FVmut}. We do not provide an example for using (iv''), but note that this result paves the way to deal with non-compact $E_X$, provided that there is a compactification $E_U$ of $O:=E_X$. In this case, we can use $E_U$ as a state space of $X$, but show that it never leaves $E_X$ using a sequence $y_1,y_2,\ldots\in E_Y$ as in Corollary~\ref{cor:htc_abc}.
\end{remark}

\subsection{Combination of mechanisms}
\label{ss.combine}
The above results develop considerable strength due to the possibility to extend the theory further to \emph{sums of operators}, each of which correspond to processes whose existence and uniqueness is already verified. Using Trotter's product formula \cite[Corollary 1.6.7]{EK86} we will show that if $E_X$ is compact and the generator can be written as a sum of operators corresponding to different mechanisms, then it suffices to check the assumptions of Theorem~\ref{T1} for each mechanism separately via Propositions~\ref{l.480a} and~\ref{l.480}, provided we have \emph{existence of the dual process corresponding to the sum}. More general state spaces will be discussed briefly in the Outlook-Section~\ref{ss.outcomp}.

\begin{theorem}[Trotters formula, combination of mechanisms]
  \label{T572}
  Let $E_X$, $E_Y$, $H$, $\mathcal H_X$, $\mathcal H_Y$ be as in Theorem~\ref{T1} and assume $\beta=0$. In addition, let $G_Y^{(1)},\dots, G_Y^{(m)}$ satisfy the conditions for $G_Y$ in Theorem~\ref{T1}(i), giving rise to  Markov processes $Y^{(1)},\dots, Y^{(m)}$ and semigroups $P^{(1)},\dots,P^{(m)}$. Assume that
  \begin{enumerate}
      \item[(a)] (i) and (ii) of Proposition~\ref{l.480} hold (in particular, $E_X$ is compact); 
      \item[(b)] (iii) and (iv) of Proposition~\ref{l.480} hold for each $Y^{(1)},\dots,Y^{(m)}$, and the corresponding semigroups $P^{(1)},\dots,P^{(m)}$; 
      \item[(c)] the $(G_Y^{(1)} + \cdots + G_Y^{(m)}, \mathcal H_Y)$-martingale problem is well-posed with solution $Y$, $\mathcal H_Y$ is a core, and the semigroup of $Y$ is $\mathcal C_b(E_Y)$-Feller.
  \end{enumerate}
  Then, there is a family of probability kernels $(\mu_t)_{t\geq 0}$ from $E_X$ to $E_X$, such that \eqref{eq:toshow0} and \eqref{eq:toshow} hold. In particular, for each $x\in E_X$ there exists a Markov process $X = (X_t)_{t\geq 0}$ starting in $x$ with transition kernels  $(\mu_t)_{t\geq 0}$ such that $X$ is Feller and the unique solution of the $(G_X^{(1)} + \cdots + G_X^{(m)}, \mathcal H_X)$-martingale problem.
\end{theorem}

If in Theorem~\ref{T572} for some reasons we already know that some of the mechanisms involved have actually unique solutions of the corresponding martingale problems then we can avoid checking the conditions for that mechanism. For an application of the strategy described in the following remark see Example~\ref{ex.1249}.

\begin{remark} [Well-posed mechanisms]
 \label{cor.814}
For Theorem~\ref{T572}(b), assume that $Y^{(1)},\dots, Y^{(m)}$ are Feller. If for some $i \in \{1,\dots,m\}$ the $G_X^{(i)}$-martingale problem is well-posed and in $H$-duality with $Y^{(i)}$, then the semigroup of $X^{(i)}$ satisfies (iv) of Proposition~\ref{l.480} by construction.
\end{remark}

\begin{remark}[Deterministic solutions of martingale problems with first order operators]
  \leavevmode \\
  The corollary is applied in Example \ref{ex:FVmut} in the case that $G_X^{(1)}$ has a simple structure. Recall \label{rem:dsomp} the conditions and notation of Theorem~\ref{T1} and assume that the  operator $G_X$ is a \emph{first order operator}, i.e.\  $\Span(\mathcal H_X)$ is closed under multiplication and 
  \begin{align}
  \label{eq:firstorder}
  G_X \Phi^2 - 2\Phi \, G_X \Phi=0 \text{ for all functions $\Phi\in
  \Span(\mathcal{H}_X)$.}    
  \end{align} 
  Then duality guarantees existence and  uniqueness of the corresponding martingale problem. Furthermore the solutions are \emph{deterministic}. In particular the duality relation of the processes reads as follows:
  \begin{align}
    \label{eq:34}
    \mathbb{E}_x [H(X_t,y)] = H(X_t,y) = \mathbb{E}_y \Bigl[H(x,Y_t)
    \exp\Bigl(\int_0^t \beta(Y_s)\,ds\Bigr)\Bigr].
  \end{align}
  Since $\Span(\mathcal{H}_X)$ is separating, the transition kernels
  $(\mu_t)_{t\ge0}$ in \eqref{eq:toshow} must satisfy
  \begin{align}
    \label{eq:35}
    \mu_t(x,\cdot) = \delta_{F_t(x)}(\cdot),
  \end{align}
  where $F_t(x)$ is the solution of the initial value problem corresponding to $G_X$, satisfying
  \begin{align}
    \label{eq:37}
    \bigl(G_X(H(\cdot,y) \bigr)(x) = \frac{d}{dt} H(F_t(x),y)\;
    \text{ and } \; F_0(x) =x.
  \end{align}
  Note that the measurability of $(t,x) \mapsto F_t(x)$ is guaranteed by \eqref{eq:34} because $Y$ is a solution of a martingale problem.
\end{remark}

\section{Proofs}
\label{s.proofresults}
\subsection{Proof of Theorem~\ref{T1}}
   \begin{proof}[Proof of Theorem~\ref{T1}(i)]
  For the semigroup property of $(P_t)_{t\geq 0}$, observe that by construction (i.e.\ linearity), and domimated convergence, using Fubini,
  \begin{align}
    \label{eq:prt21ii}
    \begin{split}
     P_t P_s H(.,y) & = \mathbb E_y\Big[ P_t H(., Y_s) \exp \Big( - \int_0^s \beta(Y_r) dr\Big) \Big]
     \\ & = \mathbb E_y\Big[ \mathbb E_{Y_s}\Big[ H(., Y_t) \exp \Big( - \int_0^t \beta(Y_r) dr\Big) \Big]\exp \Big( - \int_0^s \beta(Y_r) dr\Big) \Big] \Big]
     \\ & = \mathbb E_{y}\Big[ H(., Y_{t+s}) \exp \Big( - \int_0^{t+s} \beta(Y_r) dr\Big) \Big]
     \\ & = P_{t+s} H(.,y).
     \end{split}
  \end{align}
   For its generator $G_X$ and each $y\in E_Y$ we have
  \begin{align}
    \label{eq:131b}
    \begin{split}
      G_XH(\cdot,y)(x)
      & = \lim_{h\to 0} \frac{1}{h}\mathbb E_y\Big[H(x,Y_h)
      \exp\Big(\int_0^h\beta(Y_s)\,ds\Big) - H(x,y)\Big] \\
      & = \lim_{h\to 0} \frac{1}{h}
      \mathbb E_y\Big[H(x,Y_h)\Big(\exp\Big(\int_0^h\beta(Y_s)\,ds\Big)
      - 1\Big) + H(x,Y_h) - H(x,y)\Big] \\
      & = \beta(y) H(x,y) + G_YH(x,\cdot)(y),
    \end{split}
  \end{align}
  where we have used the strong continuity of the semigroup of $Y$, i.e.\ $Y_h \xRightarrow{h\to 0}y$. This shows that \eqref{eq:dual3new} holds.
\end{proof}

 \begin{proof}[Proof of Theorem~\ref{T1}(ii)]
  (ii) By Theorem~4.1.1 in \cite{EK86}, there exists a Markov process $X$ with transition functions $(\mu_t)_{t\geq 0}$, provided that $(\mu_t)_{t \geq 0}$ is a family of probability distributions satisfying \eqref{eq:toshow0}, $\mu_0(x,\cdot)=\delta_x(\cdot)$ and
  \begin{align}
    \label{eq:81}
    \mu_{t+s}(x,\cdot) = \int \mu_t(x,   dx')\mu_s(x',\cdot), \quad s,t\geq 0, \; x\in E_X.
  \end{align}
  First, by \eqref{eq:toshow}, there exists a transition kernel
  $\mu_0$ such that for all $y\in E_Y$
  \begin{align}
    \label{eq:6}
    H(x,y) = \int\limits_{E_X} \mu_0(x,dx') H(x',y).
  \end{align}
  Since $\Span(\mathcal{H}_X)$ is separating on $\mathcal M_1(E_X)$,
  this implies $\mu_0(x,dx') = \delta_x(dx')$.

  In order to show~\eqref{eq:81}, observe that, by \eqref{eq:toshow}, the semigroup-property of (i), and Fubini,
  \begin{align}
    \label{eq:1278b}
    \begin{aligned}
      \int \mu_{t+s}(x, dx'') H(x'',y) & = P_{t+s}H(., y)(x) = P_s (P_t H(.,y))(x) \\ & = 
      P_s \mathbb E_y \Big[ H(.,Y_t) \exp \Big( - \int_0^t \beta(Y_r)dr \Big) \Big](x) \\ & = 
      \mathbb E_y \Big[ (P_s H(.,Y_t)(x)) \exp \Big( - \int_0^t \beta(Y_r)dr \Big) \Big] \\ & = \int \mu_s(x, dx') \mathbb E_y \Big[ H(x', Y_t)\cdot \exp \Big( - \int_0^t \beta(Y_r)dr \Big) \Big] 
      \\ & = \int \mu_s(x,dx') \int \mu_t(x', dx'') H(x'', y).
    \end{aligned}
  \end{align}
  Since $\Span(\mathcal{H}_X)$ is separating, we have shown \eqref{eq:81} and we have constructed a Markov process $X$ with $X_0=x$ and
  \begin{align}
    \label{eq:0129}
    \mathbb E_x[H(X_t,y)] = \int\limits_{E_X} \mu_t(x,dx')H(x',y) = \mathbb
    E_y\Big[H(x,Y_t) \exp\Big(\int_0^t\beta(Y_s)\,ds\Big)\Big].
  \end{align}

  We now show that $X$ is the unique solution of the martingale problem for $(G_X,\mathcal H_X,x)$. Uniqueness follows directly from Proposition~4.4.7 in \cite{EK86}). If we can show that $X$ has a progressively measurable modification, then by Proposition~4.1.7 in \cite{EK86} $X$ is a solution of the martingale problem for $({\mathcal H}_X,G_X,x)$, i.e.\ existence follows. For existence of a progressively measurable modification, it remains to show $\mu_h(x,\cdot) \xRightarrow{h\to 0} \delta_x(\cdot)$ for all $x \in E_X$; see Theorem~II.2.6 in \cite{Doob1953}, together with \cite[Proposition 1.12]{KaratzasShreve1991}. This follows via duality from the Assumption~\eqref{eq:toshow} on the dual process.
  \end{proof}
 
  \begin{proof}[Proof of Theorem~\ref{T1}(iii)]
 For the $\mathcal C_b(E_X)$-Feller property, using strong continuity of the semigroup of $Y$ and  \eqref{eq:0129}, we obtain $\mathbb E_x[H(X_t,y)] \xrightarrow{t\to 0} H(x,y)$ for all $y \in E_Y$. If $\Span(\mathcal{H}_X)$ is convergence determining, this implies $X_t \xRightarrow{t\to 0}x$ and therefore $\mathbb E_x[f(X_t)] \xrightarrow{t\to 0} f(x)$ for all $f\in\mathcal C_b(E_X)$. This shows that the semigroup of $X$ is strongly continuous. In order to show continuity of $x\mapsto \mathbb E_x[f(X_t)]$ for $f\in\mathcal C_b(E_X)$ and $t\geq 0$, let $x, x_1, x_2,\ldots\in E_x$ such that $x_n \xrightarrow{n\to\infty} x$ and write $X_t^x$ for a random variable distributed according to $(X_t)_\ast \mathbb P_x$. Using dominated convergence and continuity of $H$ we obtain 
  \begin{align}
    \label{eq:736}
    \begin{aligned}
      \mathbb E[H(X^{x_n}_t,y)]
      & = \mathbb E_y\Big[H(x_n, Y_t) \exp\Big(\int_0^t
      \beta(Y_s)\,ds\Big)\Big] \\
      & \xrightarrow{n\to\infty} \mathbb E_y\Big[H(x, Y_t)
      \exp\Big(\int_0^t \beta(Y_s)\,ds\Big)\Big] = \mathbb E[H(X_t^x,y)],
    \end{aligned}
  \end{align}
  for all $y$. This shows that $X_t^{x_n} \xRightarrow{n\to\infty}  X_t^x$, since $\Span(\mathcal{H}_X)$ is convergence determining.  Therefore, for $f\in\mathcal C_b(E_X)$, $x\mapsto \mathbb  E[f(X_t^x)]$ is continuous and bounded, i.e.\ $X$ is $\mathcal  C_b(E_X)$-Feller.
\end{proof}

\subsection{Proof of Propositions~\ref{l3}, \ref{l.480}, \ref{l.480a} and Corollary~\ref{cor:htc_abc}}
\begin{proof}[Proof of Proposition~\ref{l3}]
  We first show (i) and then we prove that (ii) is a consequence of (i). Since
  $\Span(\mathcal{H}_X)$ is convergence determining, we only have to
  show that for all $y\in E_Y$ the mapping
  \begin{align}
    \label{eq:31}
    (t,x) \mapsto \int \mu_t(x, dx') H(x',y) = \mathbb E_y\Big[H(x,
    Y_t) \exp\Big(\int_0^t \beta(Y_s)ds\Big)\Big]
  \end{align}
  is continuous. Continuity in $x$ follows from boundedness of $H$ and dominated convergence. Continuity in $t$ follows from the Markov property and the strong continuity of the semigroup of $Y$.

  \noindent
  To see that (i) implies (ii) let
  $f_1, f_2,\dots \in \mathcal C_b(E_X)$ be such that
  $f_n \xrightarrow{n\to\infty} \indset\Gamma$ boundedly pointwise,
  then $(t,x)\mapsto \mu_t(x,\Gamma) = \lim_{n\to\infty}\int \mu_t(x, dx') f_n(x')$ is measurable as a limit of continuous functions.
\end{proof}

\begin{proof}[Proof of Proposition~\ref{l.480}]
  First, $\Span(\mathcal H_X)$ is dense in $\mathcal C(E_X)$ due to (i), (ii) and the Stone-Weierstrass theorem. Second, $P_t^x$ is continuous (since the semigroup of $Y$ is $\mathcal C_b(E_Y)$-Feller by (iii)) on $\Span(\mathcal H_X)$, so there is a unique extension of $P_t^x$ to $\mathcal C(E_X)$. By continuity, this extension also satisfies (iv), i.e.\ it is a  positive linear form.
  By the Riesz-Markov theorem, there is a unique measure $\mu_t(x,\cdot)$ such that
  \begin{align}
    \label{eq:39}
    P_t^x f = \int_{E_X} \mu_t(x,dx') f(x'), \quad f\in {\mathcal C}(E_X).
  \end{align}
  Since $P^x_t1 = 1$ by (iv), we know that $\mu_t(x,\cdot) \in \mathcal M_1(E_X)$. Applying \eqref{eq:39} to $\mathcal H_X$ and using \eqref{eq:16}
  we obtain
  \begin{align}
    \mathbb E_{y}\Bigl[H(x,Y_t)
    \exp\Bigl(\int_0^t \beta(Y_s)\, ds\Bigr)\Bigr] = \int_{E_X}
    \mu_t(x,dx') H(x',y),
  \end{align}
  which is precisely~\eqref{eq:toshow}. Since $x\mapsto \mu_t(x,\cdot)$ is
  continuous by the last display, we conclude that $\mu_t$ is a
  probability kernel from $E_X$ to $E_X$.
\end{proof}

\begin{proof}[Proof of Corollary~\ref{cor.480}]
    Let $P^x$ be as in Theorem~\ref{T1}(i). By the approximate duality, we can write for \eqref{eq:16} using the convergence $Y_n\xRightarrow{n\to\infty} Y$
    \begin{equation}
        \begin{aligned}
        P_t^xH(.,y) & = \lim_{n\to\infty} 
        \mathbb E_y\Big[ H(x,Y^n_t) \exp\Big( \int_0^t \beta(Y^n_s)ds\Big) \Big] = \lim_{n\to\infty} 
        \mathbb E_x[ H(X_t^n, y)].
        \end{aligned}
    \end{equation}
    This calculation shows that $P_t^x$ from Proposition~\ref{l.480} is positive (as a limit of positive maps) and $P_t^x1 = 1$ (as a limit of maps with the same property). Hence, Proposition~\ref{l.480}(iv) holds. 
\end{proof}

\begin{proof}[Proof of Proposition~\ref{l.480a}]
  Let $t\geq 0$ and $x\in E_X$. First, since $E_Y$ is dense in $\mathcal C(E_U)$ and the semigroup of $Y$ is $\mathcal C_b(E_Y)$-Feller, we can extend $Q_t^x$ to $\mathcal C(E_U)$. This then gives a positive linear form on $\mathcal C(E_U)$. By the Riesz-Markov theorem (recall that $E_U$ is  compact), we find a probability measure  $\nu_t(x,\cdot) \in \mathcal M_1(E_U)$ such that  $Q_t^x f = \int \nu_t(x,du) f(u)$ for all $f\in\mathcal C(E_U)$. For $y\in E_Y$, this amounts to
  \begin{align}\label{eq:310}
    \int \nu_t(x, du) y(u) = \mathbb E_{y}\Bigl[H(x,Y_t)
    \exp\Bigl(\int_0^t \beta(Y_s)\, ds\Bigr)\Bigr].
  \end{align}
  Now, let $U$ have the distribution $\nu_t(x,\cdot)$. By assumption,
  we find $X$ so that $\mathbb E[H(X,y)] = \mathbb E[y(U)]$ for all $y\in E_Y$. Denoting the distribution of $X$ by $\mu_t(x,\cdot)$, we
  obtain
  \begin{align*}
    \int \mu_t(x, dx') H(x', y) = \int \nu_t(x, du) y(u)
  \end{align*}
  and \eqref{eq:toshow} follows.
\end{proof}

\begin{proof}[Proof of Corollary \ref{cor:htc_abc}]
From the proof of Proposition~\ref{l.480a}, we see that we only need to find $X$ such that \eqref{eq:abc} holds for $U \sim \nu_t(x,\cdot)$, where $\nu_t$ is from \eqref{eq:310}. 

Under the assumptions of (iv'), we claim that for all $f\in\mathcal F$ and $y\in E_Y$ \begin{align}
    \label{eq:cni}
      \mathbb E\bigl[H(x, Y_t^{y\circ f})\bigr]
        = \mathbb E\bigl[H(x, Y_t^{y})\bigr].
\end{align}
Indeed, by assumption, using the semigroup $(S_t)_{t\geq 0}$ of $Y$,
\begin{align} \notag
  \mathbb E\bigl[H (x, Y_t^{y\circ f}) \bigr] & = H(x, y\circ f) + \int_0^t S_s G_Y H(x, y\circ f) \\ & = H(x, y) + \int_0^t S_s G_Y H(x,y)ds = \mathbb E\bigl[H (x, Y_t^{y}) \bigr]. \label{eq:019}
\end{align}
Then, we have for such all $U$ with $U\sim \nu_t(x,\cdot)$ and for all $y\in E_Y$, 
  \begin{equation}
    \label{eq:mnc}
    \begin{aligned}
      \mathbb E[y(f(U))] & = \int \nu_t(x,du) (y\circ f)(u) = \mathbb
      E\bigl[H(x, Y_t^{y\circ f}) \Bigr] = \mathbb E\bigl[H(x, Y_t^{y})
      \Bigr] = \mathbb
      E[y(U)].
    \end{aligned}
  \end{equation}
  Hence, $U \sim f(U)$ for all $f\in\mathcal F$, so we require to show \eqref{eq:abc} only for such $U$.

  \medskip
  \noindent
  For (iv''), we write for $U\sim \nu_t(x,\cdot)$
  \begin{equation}
    \label{eq:cwo}
    \begin{aligned}
      \mathbb P(U\in O^c) & = \lim_{n\to\infty} \mathbb E[y_n(U)] =      \lim_{n\to\infty}\int \nu_t(x,du) y_n(u) \\ & =      \lim_{n\to\infty}\mathbb E\bigl[H(x, Y_t^{y_n})      \exp\Bigl(\int_0^t \beta(Y_s^{y_n})\, ds\Bigr)\Bigr] = 0.
    \end{aligned}
  \end{equation}
  Hence, $U$ has values in $O$, almost surely, and we need to show~\eqref{eq:abc} only for such $O$.
\end{proof}

\subsection{Proof of Theorem~\ref{T572}}
  To facilitate reading, we restrict ourselves to the case $m=2$. We will use Trotter's product formula for the semigroups $(Q^{(1)}_t)_{t\geq 0}$, $(Q^{(2)}_t)_{t\geq 0}$ and  $(Q_t)_{t\geq 0}$, given by
  \begin{align*}
    Q^{(i)}_t f(y)
    & = \mathbb E_y[f(Y^{(i)}_t)], \quad i=1,2,\qquad Q_t f(y)
    = \mathbb E_y[f(Y_t)].
  \end{align*}
  These are strongly continuous contraction semigroups on $\mathcal C_b(E_Y)$ with generators $G_Y^{(1)}, G_Y^{(2)}$ and $G_Y^{(1)} + G_Y^{(2)}$, respectively. In addition, $\mathcal H_X$ is a core for $G_Y^{(1)} + G_Y^{(2)}$ by assumption. From Trotter's formula, we see that 
  \begin{align}
    \label{e607}
    Q_tf=\lim_{n \to \infty} (\prescript{(n) \mkern-2mu}{}Q_tf),
    \text{  with  }\prescript{(n) \mkern-2mu}{}
    Q_t=\left(Q_{t/2n}^{(2)} Q_{t/2n}^{(1)} \ldots Q_{t/2n}^{(2)}
      Q_{t/2n}^{(1)}\right) \text{ and } 2n-\text{factors.}
  \end{align}
  From (a) and (b) and Propositions~\ref{l3} and \ref{l.480} we know that  \eqref{eq:toshow0} and \eqref{eq:toshow} hold for $i=1,2$. So, as Theorem~\ref{T1} shows, there are $E_X$-valued processes $X^{(1)}$ and $X^{(2)}$ with semigroups $P^{(1)}$ and $P^{(2)}$, respectively, given by
  \begin{align*}
    P_t^{(i)}H(\cdot,y)(x) = \mathbb E_x[H(X_t,y)] =
    Q_t^{(i)}H(x,\cdot)(y), \quad i=1,2.
  \end{align*}
  As a next step, we define
  \begin{align*}
    P_tH(\cdot,y)(x)
    & \coloneqq Q_tH(x,\cdot)(y) = \lim_{n \to \infty}
      \left(Q_{t/2n}^{(2)} Q_{t/2n}^{(1)} \ldots Q_{t/2n}^{(2)}
      Q_{t/2n}^{(1)}\right)H(x,\cdot)(y)
    \\ & = \lim_{n\to\infty} \left(P_{t/2n}^{(2)} P_{t/2n}^{(1)}
         \ldots P_{t/2n}^{(2)} P_{t/2n}^{(1)}\right) H(\cdot,y)(x).
  \end{align*}
  By Proposition~\ref{l.480}, $P^{(1)}_t$ and $P^{(2)}_t$ can be extended to $C(E_X)$ and by (b) are positive with $P^{(1)}_t 1 = P^{(2)}_t 1 = 1$. So,  we see that $P_t$ can be continuously extended on $\mathcal C(E_X)$ with $P1=1$ and by the Riesz-Markov Theorem, for every $x\in E_X$ and $t\geq 0$, there is a Markov kernel $\mu_t(x,.)$ from $E_X$ to $E_X$ such that \eqref{eq:toshow} holds. In addition, \eqref{eq:toshow0} holds since $(t,x)\mapsto\mu_t(x,.)$ is measurable as a limit of continuous functions; see Proposition~\ref{l3}. Hence, all conclusions of Theorem~\ref{T1}(ii) follow. \hfill \qed

\section{Examples}
\label{s.examples}
In this section we give several examples how the above results can be applied. We will  distinguish between the compact and locally compact case. Example~\ref{ex:FV} is the Fleming-Viot measure-valued diffusion (without mutation and selection), which is a process taking values in $\mathcal M_1([0,1])$. As an extension, we consider the Cannings model in Example~\ref{ex.can-mod} with the same state space, but \cadlag paths with jumps. In Example~\ref{ex.1249}, we add a spatial component, which gives an application of Theorem~\ref{T572}. In Example~\ref{ex:FVmut} (Fleming-Viot process with mutation) and Example~\ref{ex:slvf} (spatial $\Lambda$-Fleming-Viot process), we use function-valued duals. Turning to the case of locally compact state spaces, we treat in Example~\ref{ex:CSBP} the continuous state branching process and in Example~\ref{ex:Feller} as a special case the Feller branching process using a different duality, referred to as Feynman-Kac duality, and $\beta\neq 0$.

\subsection{Compact state spaces -- resampling systems}
\label{ss.compstat}
Population models with a constant population size do not only arise frequently in population genetics, but are also frequently analysed using dual processes. Their large-population-limits come as solutions of stochastic differential equations, such as the Wright-Fisher diffusion, measure-valued diffusions, or more complex approaches, such as historical or tree-valued processes; see e.g.\ \cite{DP91, D93, Etheridge2000, GPWmp13, DGP12}. In this section, we remain in the realm of compact state spaces and apply our abstract results to the measure-valued Fleming-Viot process without (Example~\ref{ex:FV}) and with mutation (Example~\ref{ex:FVmut}). The moment-duality for the former example is extended in Example~\ref{ex.can-mod} to more general Cannings models, and in Example~\ref{ex.1249} to a geographically structured model, studied with results from Section~\ref{ss.combine}. Then, we will have a function-valued dual in the Fleming-Viot process with mutation from Example~\ref{ex:FVmut} and the spatial Lambda-Fleming-Viot process from Example~\ref{ex:slvf}. 

For Examples~\ref{ex:FV}--\ref{ex:FVmut}, we have $E_X = \mathcal M_1(I)$ with $I=[0,1]$, and recall that $F : \mathcal M_1([0,1]) \to \R$ is called differentiable if for all $x \in \mathcal M([0,1])$ (the set of finite Borel-measures on $[0,1]$) and $u \in [0,1]$
\begin{align}
    \label{eq:24}
    \frac{\partial F (x)}{\partial x} [u] = \lim_{\varepsilon
    \downarrow 0} \frac{F(x+\varepsilon \delta_u) - F(x)}{\varepsilon}
    \; \text{ exists and } \;
    (u,x) \mapsto \frac{\partial F (x)}{\partial x} [u] \; \text{ is
    continuous}.
\end{align}
In the obvious way, if it exists, the second derivative is defined as
\begin{align}
    \label{eq:25}
    \frac{\partial^2 F(x)}{\partial x \partial x}[u,v] \coloneqq
    \frac{\partial}{\partial x} \biggl(\frac{\partial F
    (x)}{\partial x} [u] \biggr) [v]
\end{align}
and we say that $F$ is twice continuously differentiable if $(u,v,x) \mapsto \frac{\partial^2 F(x)}{\partial x \partial x}[u,v]$ is continuous.  We set
\begin{align}
    \label{eq:17}
    \mathcal D_X = \{ F : \mathcal M_1([0,1]) \to \R: F \text{ is twice
    continuously differentiable} \}.
\end{align}

\subsubsection{Resampling systems with compact state spaces and particle-valued duals}
\begin{example}[Fleming-Viot process]
  \label{ex:FV}
  For $F\in \mathcal D_X$, we define
  \begin{align}
    \label{eq:16a}
    G_X F(x) & \coloneqq  \int_{[0,1]} \int_{[0,1]} \frac{\partial^2
               F(x)}{\partial x \partial x}[u,v] (x(du)
               \delta_u(dv) - x(du)x(dv)).
  \end{align}
  For $u = (u_1,\dots,u_n) \in [0,1]^n$, $k = (k_1,\dots,k_n) \in
  \N_0^n$ we write $u^{k} \coloneqq u_1^{k_1} \cdots u_n^{k_n}$. Then for $F(x) \coloneqq \langle x^{\otimes n}, u^{k} \rangle \coloneqq \int x^{\otimes n}(du) \, u^{k}$, a straight-forward calculation shows that
  \begin{align}\label{eq:16b}
    G_X F(x) & \coloneqq  \sum_{1\leq i < j \leq n} \langle
               x^{\otimes n}, u^{\theta_{ij} k} - u^k\rangle,
  \end{align}
  where $\theta_{ij} k \in \N_0^{n-1}$ arises from $k=(k_1,\dots,k_n)$ by replacing $k_{i\wedge j}$ by $k_i + k_j$, and shifting the indices above $i\vee j$ down by one; see \cite[p.~31]{D93}.

  The \emph{dual process} $Y$ is the partition-valued Kingman coalescent, i.e.\ a pure jump process in which every pair of partition elements coalesces at rate $1$. More formally, we take
  \begin{align}
    \label{eq:132}
    E_Y \coloneqq \bigcup_{n=0}^\infty \mathcal P_n,
  \end{align}
  where $\mathcal P_n$ is the set of partitions of $\{1,\dots,n\}$. For a partition $y$ we write $|y|$ for the number of its partition elements of $y$. A partition $y \in \mathcal P_n$ induces an equivalence relation ``$\sim$'' on $\{1,\dots,n\}$ with $i \sim j$ if $i$ and $j$ are in the same partition element of $y$. We order  the partition elements of $y \in \mathcal P_n$ according to their smallest elements, so that we have a well-defined representation $y = \{y_1,\dots,y_{|y|}\}$. We also write $|y_k|$ for the number of elements of the $k$th partition element and we write $y(i) = k$ if $i\in y_k$, i.e.\ $y(i)$ is the number of the partition element $i$ is in. To define the duality function for $y\in \mathcal P_n$ and $u = (u_1,\dots,u_{|y|}) \in [0,1]^{|y|}$, we set
  \begin{align}
    \label{eq:133}
    u^y \coloneqq \prod_{k=1}^{|y|} u_k^{|y_k|} = \prod_{i=1}^n
    u_{y(i)}, \qquad H(x,y) = \int x^{\otimes |y|}(du) \, u^y.
  \end{align}
  For the dynamics of the dual, for $y \in \mathcal P_n$ with $|y|=m\leq n$ and $y' \in \mathcal P_m$, we write $y'\diamond y \in \mathcal P_n$ for the partition with $i \sim j$ if $y'(y(i)) = y'(y(j))$, $i,j=1,\dots,n$. In other words $y'\diamond y$ arises from $y$ by merging partition elements of $y$ according to partition $y'$. For example, if $y=\{\{1\},\{2,3\},\{4\}\}$ and $y'=\{\{1,3\},\{2\}\}$, then $y'\diamond y = \{\{1,4\},\{2,3\}\}$.  For $1\le i<j \le m$ we define the partition $\theta_{\{ij\}} \in \mathcal P_{m}$ by (note that $i\in \{1,j-1\}$ and $j=m$ is possible below)
  \begin{align}
    \label{eq:15}
    \theta_{\{ij\}} \coloneqq \theta_{\{ij\}}^{\,(m)} \coloneqq
    \{\{1\},\dots,\{i-1\}, \{i,j\},\{i+1\},\dots,
    \{j-1\},\{j+1\} \dots,\{m\}\} \in \mathcal P_{m-1}.
  \end{align}
  We will typically omit the dependence of $\theta_{\{ij\}}$ on $m$ but it should be clear from the context what $m$ is. For instance, the  operation $\theta_{\{ij\}} \diamond y$ means that $m=|y|$ and that $i$th and $j$th partition elements of $y$ are merged.

  With this notation the process $Y$ can be defined as a process with transitions
  \begin{align}
    \label{eq:135}
    \text{from } Y = y \text{ to } \theta_{\{ij\}} \diamond y \text{ at
    rate $1$ for all $1\leq i<j\leq |y|$}.
  \end{align}
  Thus, $Y$ solves the martingale problem with the operator $G_Y$ acting on the duality function as follows
  \begin{align}
    \label{eq:136}
    G_YH(x,\cdot)(y) = \sum_{1\leq i<j\leq |y|} \langle x^{\otimes(|y|-1)},
    u^{\theta_{\{ij\}} \diamond y}\rangle - \langle x^{\otimes |y|}, u^y\rangle
    = \sum_{1\leq i<j\leq |y|} \langle x^{\otimes |y|},
    u^{\theta_{\{ij\}} \diamond y} - u^y\rangle,
  \end{align}
  hence $G_X H(\cdot,y)(x) = G_Y H(x,\cdot)(y)$ by \eqref{eq:16b}.  We note that $Y$ is a Feller process and writing $\mathbb P_m$ for the  distribution with initial condition $\{\{1\},\dots,\{m\}\}$,
  \begin{align}
    \label{eq:7}
    \mathbb E_y[f(Y_t)] = \mathbb E_{|y|}[f(Y_t \diamond y)]
  \end{align}
  by the definition of $Y$ (since the dynamics is on and not within the partition elements). For $y \in \mathcal P_n$ we have
  \begin{align*}
    P_tH(\cdot,y)(x)
    & \coloneqq \mathbb E_y[H(x, Y_t)]
      = \sum_{y'} \mathbb P_y[Y_t = y'] \int_{[0,1]^{|y'|}} x^{\otimes
      |y'|}(du_1,\dots,du_{|y'|})\, u_1^{|y_1|} \cdots u_{|y'|}^{|y'_{|y'|}|},
  \end{align*}
  where the sum is over all $y' \in \mathcal P_m$, $m \in
  \{1,\dots,n-1\}$ which are possible outcomes (otherwise the probability is
  $0$) of the process $Y$ with initial condition $y$.

  To check the condition \eqref{eq:toshow} of Theorem~\ref{T1} we need to find a $\mathcal M_1([0,1])$-valued random variable $X_t$ such
  that for all $y\in E_Y$,
  \begin{align}
    \label{eq:toshowWF}
    \begin{split}
      \mathbb E_x[\langle X_t^{\otimes |y|}, u^y\rangle] & = \mathbb
      E_x\Big[\int_{[0,1]^{|y|}} X_t^{\otimes |y|}(du_1,\dots,du_{|y|})
      u_1^{|y_1|} \cdots u_{|y|}^{|y_{|y|}|} \Big] \\
      & = \mathbb E_y[\langle x^{\otimes |Y_t|}, u^{Y_t}\rangle] =
      \mathbb E_{|y|}[\langle x^{\otimes |Y_t|}, u^{Y_t \diamond y}
      \rangle] \eqqcolon \mathfrak{m}_y.
    \end{split}
  \end{align}
  In order to find $X_t$, we first fix $m \in \N$. We need to find $[0,1]$-valued random variables $U_1,\dots,U_m$ such that for $U=(U_1,\dots,U_m)$ all $y\in E_Y$ with $|y|=m$
  \begin{align}
    \label{eq:U}
    \mathbb E[U^y] = \mathfrak{m}_y.
  \end{align}
  By the \emph{multi-dimensional Hausdorff moment problem} \cite[Proposition~6.11, p.~134]{Berg1984}, this is guaranteed given
  that, for all $k, \ell \in \mathbb N_0^m$ (and
  sum over all $p = (p_1,\dots, p_m) \in \mathbb N_0^m$ with $p \le \ell$ componentwise below) we have
  \begin{align}
    \label{eq:138}
    \sum_{p_1=0}^{\ell_1} \cdots \sum_{p_m=0}^{\ell_m} (-1)^{p_1 +
    \cdots +p_m} \binom{\ell_1}{p_1}
    \cdots \binom{\ell_m}{p_m} \, \mathfrak{m}_{y^{(k+p)}} \geq 0.
  \end{align}
  Here, for $k =(k_1,\dots,k_m) \in \N_0^m$, $y=y^{(k)}$ is a partition with $|y|=m$ and $|y_1| = k_1,\dots, |y_m|=k_m$. Note that for $p, k \in \N_0^m$ and $y \in \mathcal P_m$ we have $u^{y\diamond y^{(k+p)}} = u^{y\diamond y^{(k)}} u^{y\diamond y^{(p)}}$.

  In order to show \eqref{eq:138},  we write for the left hand side
  \begin{align}
    \label{eq:139}
    \begin{aligned}
      \sum_{p_1=0}^{\ell_1} \cdots \sum_{p_m=0}^{\ell_m}
      & (-1)^{p_1 + \cdots +p_m} \binom{\ell_1}{p_1} \cdots \binom{\ell_m}{p_m}
        \mathbb E_{m}[\langle x^{\otimes |Y_t|}, u^{Y_t \diamond
        y^{(k+p)}}\rangle] \\
      & = \mathbb E_m\Big[\langle x^{\otimes
        |Y_t|}, u^{Y_t \diamond y^{(k)}} \sum_{p_1=0}^{\ell_1} \cdots
        \sum_{p_m=0}^{\ell_m} (-1)^{y^{(p)}} \binom{\ell_1}{p_1} \cdots
        \binom{\ell_m}{p_m} u^{Y_t \diamond y^{(p)}}\rangle\Big] \\
      & = \mathbb E_m[\langle x^{\otimes |Y_t|}, u^{Y_t\diamond
        y^{(k)}}(1-u)^{Y_t\diamond y^{(\ell)}}\rangle] \geq 0.
    \end{aligned}
  \end{align}
  Hence we have shown the existence of $U_1,\dots,U_m$ with \eqref{eq:U}. By a projective limit argument we can extend this construction to show existence of $U_1, U_2,\dots$ such that \eqref{eq:U} holds for any finite subset. The resulting sequence $U_1,U_2,\dots$ is exchangeable. Hence, by de Finetti's theorem there is a $\mathcal M_1([0,1])$-valued random variable $X_t$ such  that $U_1,U_2,\dots$ is independent given $X_t$. In particular we have
  \begin{align}
    \label{eq:140}
    \mathbb E[U^y] = \mathbb E[\langle X_t^{\otimes |y|}, u^y\rangle].
  \end{align}
  But this is exactly \eqref{eq:toshow}. Moreover, \eqref{eq:toshow0} holds by Proposition~\ref{l3}. Hence, well-posedness of the $(G_X, \mathcal H_X)$-martingale problem follows. Since $Y$ is Feller, and $\mathcal H_X$ is convergence determining (since $[0,1]$ is  compact), $X$ is Feller as well by Theorem~\ref{T1}(ii). By  Proposition~\ref{p.loccomp-cadl}, there is a modification with \cadlag paths. Moreover, since $G_X$ is second order, the solution has a version with almost surely continuous paths; see  Proposition~4.5 in \cite{DGP12}. This process is usually referred to as the \emph{Fleming-Viot measure-valued process}.
\end{example}

\begin{example}[Cannings model]
  \label{ex.can-mod}
  For the Cannings model (without Fleming-Viot resampling), let $\Lambda$ be a finite measure on $(0,1]$ (implying no mass at~0) and  $\Lambda^*(dr) = \frac{\Lambda(dr)}{r^2}$, $r \in (0,1]$. Here,
  \begin{align}
    \label{eq:28}
    G_X F(x) = \int_{(0,1]} \Lambda^*(dr) \int_{E_X} x(du)
    \bigl(F((1-r)x + r\delta_u) -F(x)\bigr).
  \end{align}
  We note that for $F(x) = \langle x^{\otimes n}, u^k\rangle$, we have
  \begin{align}\label{eq:29}
    G_X F(x) = \sum_{j=2}^n \binom nj \lambda_{nj}
    \bigl(\langle x^{\otimes(n-j+1)},
    u^{\theta_{j} k}\rangle - \langle x^{\otimes n},
    u^k\rangle\bigr)
  \end{align}
  with $\theta_{j} k = (k_1,\dots,k_{n-j}, k_{n-j+1} + \cdots + k_n)$ (with a slight abuse of notation for $\theta$ from Example~\ref{ex:FV})
  and
  \begin{align}
    \label{eq:30}
    \lambda_{nj} = \int_0^1 r^j (1-r)^{n-j} \, \Lambda^*(dr), \quad j
    \in \{2,\dots,n\}.
  \end{align}
  For the dual process $Y$, we use the notation of   Example~\ref{ex:FV}. Again, the state space of $Y$ is $E_Y$ from \eqref{eq:132}, and the duality function $H$ is as in \eqref{eq:133}. Here, $Y$ is the \emph{partition valued $\Lambda$-coalescent}, i.e.\ a pure jump process with the following dynamics: If the current state of the process consists of $n \ge 2$ blocks then each $j$-tuple merges into a single block at  rate $\lambda_{nj}$.  Let $y \in E_Y$ with $|y|=n$ and let $J \subset \{1,\dots,n\}$.  Generalizing the notation from \eqref{eq:15} and \eqref{eq:29}, we write $\theta_{J}$ for the partition of $\{1,\dots,n\}$ in which all elements of $J$ are put in a one block and the other partition elements are singletons. Note that  $\theta_{ij}$ from \eqref{eq:15} equals $\theta_{J}$ with $J=\{i,j\}$ and $\theta_k$ from \eqref{eq:29} equals $\theta_{|J|}$. With this notation the process $Y$ can be defined as a process with transitions
  \begin{align}
    \label{eq:32}
    \text{from } Y = y \text{ to } \theta_{J} \diamond y \text{ at
    rate $\lambda_{|y|,|J|}$ for $J \subset \{1,\dots,|y|\}$}.
  \end{align}
  In particular, $Y$ solves the martingale problem for
  \begin{align}
    \label{eq:33}
    \begin{split}
      G_Y H(x,\cdot)(y)
      & = \sum_{J \subset \{1,\dots,|y|\}}
      \lambda_{|y|,|J|}\bigl(\langle x^{\otimes(|y|-|J|+1)},
      u^{\theta_{J} \diamond y}\rangle - \langle x^{\otimes |y|},
      u^y\rangle\bigr)
       = \sum_{J \subset \{1,\dots,|y|\}} \lambda_{|y|,|J|} \langle
      x^{\otimes |y|}, u^{\theta_{J} \diamond y} - u^y\rangle.
    \end{split}
  \end{align}
  Now, we can argue as in Example~\ref{ex:FV} that \eqref{eq:7} also holds and that the proof that the multi-dimensional Hausdorff moment problem has a solution (see \eqref{eq:139}) literally carries through. As a result it follows that $(G_X,\mathcal H_X)$-martingale problem is well-posed. Again, $X$ is Feller since $Y$ is Feller and $\mathcal H_X$ is convergence determining, and $X$ has \cadlag paths by Proposition \ref{p.loccomp-cadl}. Here, $G_X$ is not second order. In particular, the paths are not continuous but are jump  processes which even have countably many jumps, if  $\Lambda^\ast([0,\delta])>0$ for all $\delta>0$.
\end{example}

\subsubsection{Resampling systems with compact state spaces and function-valued duals}
\label{ss.compstat.fkt}
\begin{example}[Fleming-Viot process with mutation]\label{ex:FVmut}
  Now, we add mutation to the Fleming-Viot process as introduced in Example~\ref{ex:FV}, which requires a different dual process. More precisely, the mutation operator reads for $F\in\mathcal D_X$ as in \eqref{eq:17}
  \begin{align}\label{eq:18a}
    G_X^{\rm mut}F(x)= \vartheta \int\limits^1_0 \frac{\partial F}{\partial x}
    [u] \left(\beta(u,dv)-x(du)\right),
  \end{align}
  where we refer to $\vartheta \ge 0$ as the \emph{mutation rate}, and  $\beta(\cdot,\cdot)$ is a stochastic kernel on $I$, denoting the  \emph{mutation kernel}.

  The operator of the process we aim to show existence for is given by
  \begin{align}
    \label{eq:18}
    G_X=G_X^{\rm res} + G_X^{\rm mut},
  \end{align}
  where $G_X^{\rm res}$ is given by the right hand side of \eqref{eq:16a} and $G_X^{\rm mut}$ is as in \eqref{eq:18a}. We will use here, different from the previous examples, a \emph{function-valued} dual process, i.e.\ a process with state space
  \begin{align}
    \label{eq:132b}
    E_Y \coloneqq \bigcup_{n=0}^\infty \Pi_n, \qquad \Pi_n \coloneqq
    \mathcal C(I^n),
  \end{align}
  where $\Pi_0$ consist of all constants. Moreover, we view $y \in \Pi_n$ as a continuous function with domain $I^{\mathbb N}$, depending only on $n$ coordinates, i.e.\ $E_Y \subseteq \mathcal  C(E_U)$ if we choose now $E_U = I^{\mathbb N}$. For $y\in E_Y$, we write $|y| = n$ if $y\in \mathcal C(I^n)$.

  For the \emph{duality function}, we set
  \begin{align}
    \label{eq:134b}
    H(x,y) \coloneqq \langle x^{\otimes \mathbb N},
    y\rangle \coloneqq \langle x^{\otimes |y|}, y\rangle \coloneqq
    \int x(du_1) \cdots x(du_{|y|}) \,y(u_1,\dots,u_{|y|}).
  \end{align}
  So, in words, $H(x,y)$ is computed by choosing elements $u_1,\dots,u_{|y|}$ independently from $x$, and evaluating them according to the function $y$. Setting   \begin{align}
      \label{eq:135c}
      \theta_{ij}(u_1,\dots,u_n) = (u_1,\dots, u_i,\dots,u_{j-1}, u_i,
      u_{j+1}, \dots, u_{n})
    \end{align}
    (and note that $i\in \{1,j-1\}$ and $j=n$ is possible  and again we abuse notation for $\theta$ from Examples~\ref{ex:FV} and~\ref{ex.can-mod}). Note that $G_X$ takes the special form, when applied to $H(.,y) \in \mathcal H_X$, 
    \begin{align}
      \label{eq:136new}
      \begin{split}
        G_XH(\cdot,y)(x)
      & = \sum_{1\leq i<j\leq |y|} \langle x^{\otimes \mathbb N},
        y\circ \theta_{ij} - y\rangle + \vartheta \sum_{1\leq k\le |y|}
        \langle x^{\otimes \mathbb N}, \beta_k y - y\rangle.
      \end{split}
    \end{align}
  Next, we claim that $\mathcal H_X$ is a convergence determining set of functions, and $1\in\mathcal H_X$. For this, recall that by Le Cam’s theorem \cite{LC57} (see also \cite{LoehrRippl2016}), the set of functions $\mathcal{H}_X \subseteq \mathcal C_b(E_X)$ on a completely regular Hausdorff space $E_X$ is convergence determining for Radon probability measures, if it is multiplicatively closed and induces the topology of $E_X$. In our case, $\mathcal{H}_X \coloneqq \{x\mapsto \langle x^{\mathbb N}, y\rangle: y\in E_Y\} \subseteq \mathcal C_b(E_X)$ is multiplicatively closed and for $x,x_1, x_2,... \in E_X$
  \begin{align}
    \label{eq:36}
    (x_n \xrightarrow{n\to\infty} x)
    \iff (x_n^{\otimes \mathbb N} \xrightarrow{n\to\infty} x^{\otimes \mathbb N} ) \iff ( \langle x_n^{\otimes \mathbb N}, y\rangle \xrightarrow{n\to\infty} \langle x^{\otimes \mathbb N}, y\rangle \text{ for all $y\in E_Y$}).
  \end{align}
  Hence, $\mathcal{H}_X$ induces the weak topology on $E_X$ and Le Cam's theorem implies that $\mathcal{H}_X$ is convergence determining.

  For the \emph{dynamics of the dual process}, let $Y$ be the Markov jump process, which jumps from $Y_t=y$ 
  \begin{enumerate}[(i)]
  \item for all $1\leq i<j\leq |y|$,
    \begin{align}
      \label{eq:135b}
      \text{ to } \; y \circ \theta_{ij}\text{ at
      rate $1$},
    \end{align}
    noting that $|y\circ \theta_{ij}| = |y|-1$;
  \item for all $1\leq k \leq |y|$ 
    \begin{align}
      \label{eq:135d}
      \text{ to } \; \beta_k y \text{ at rate $\vartheta$},
    \end{align}
    where $\beta_k y(u_1,\dots,u_n) \coloneqq \int
    y(u_1,\dots,u_{k-1},v,u_{k+1},\dots,u_n) \, \beta(u_k,dv)$.
  \end{enumerate}
  Since $Y$ is a pure jump process with bounded jump rates, $Y$ is the unique solution of the martingale problem $G_Y$ for
  \begin{align}
    G_Y H(x,.)(y) & := G_X H(.,y)(x),
  \end{align}
  where the right hand side is from~\eqref{eq:136new}. This already shows that \eqref{eq:dual3new} holds (with $\beta=0$).

  In order to apply Theorem~\ref{T1}, we start with \eqref{eq:toshow} using Proposition~\ref{l.480a} and check the assumptions (i)-(iv) made there. Since $Y$ is Feller, and $E_Y$ is dense in $\mathcal C(E_U)$ as an algebra containing~1 due to the Stone-Weierstrass theorem, (i) and (iii) hold. For (iv), i.e.\ the positivity of $P_t^x$, let us have a closer look at the two possible transitions of $y$ from above. If $y \geq 0$, note that $y\circ \theta_{ij} \geq 0$ and $\beta_k y \geq 0$. Writing $Y^y$ for the process $Y$ when started in $y$, and looking at the transitions of $Y^y$, it is clear that $y\mapsto Y_t^y$ is linear and $Y_t^y \geq 0$ as well, and consequently $y\mapsto \mathbb   E[H(x,Y_t^y)]$ is a positive linear form. In addition, if $y=1$, then $Y_t^y=1$, so $\mathbb E[H(x,Y_t^1)]=1$, which shows that all properties of (iii) hold.

  For (iv), we will make use of the reformulation given in Corollary~\ref{cor:htc_abc}(iv') to verify \eqref{eq:abc}. Therefore, we define the set of permutations on $\mathbb N$
  \begin{align}
    \label{eq:Sigma}
    \Sigma = \bigcup_{n=0}^\infty \Sigma_n,
    \qquad \Sigma_n \coloneqq \{\sigma: \mathbb N\to\mathbb N \text{ bijective}, \sigma|_{\{n, n+1,\dots\} = \operatorname{id}}\}.
  \end{align}
  and the set of functions
  \begin{align}
    \label{eq:mathcalF1}
    \mathcal F = \{f_\sigma: E_U \to E_U, u \mapsto u_\sigma \text{ with }
    \sigma\in \Sigma, f \in C_b(U)\},
  \end{align}
  where $u_\sigma = (u_{\sigma(1)}, u_{\sigma(2)},\dots)$. Since $x^{\otimes \mathbb N}$ is exchangeable, we have that $H(x,y\circ f_\sigma) = H(x,y)$ as well as
  \begin{equation}
    \begin{aligned}
      G^{\text{res}}_YH(x,\cdot)(y\circ f_\sigma) =
      G^{\text{res}}_YH(x,\cdot)(y), \quad G^{\text{mut}}_YH(x,\cdot)(y\circ
      f_\sigma) = G^{\text{mut}}_YH(x,\cdot)(y).
    \end{aligned}
  \end{equation}
  Let us turn to the proof of \eqref{eq:abc} for an $E_U$-valued random varibale $U$ with $U\sim f(U)$ for all $f\in\mathcal F$ and $E_U$-valued random variable. In other words, $U$ is exchangeable and using the de Finetti's theorem we obtain for such $U$ that there is an $E_X = \mathcal M_1(I)$-valued random variable $X$ such that $U \sim X^{\mathbb N}$ conditional on $X$. In other words, \eqref{eq:abc} holds. Hence, Proposition~\ref{l.480a} gives \eqref{eq:toshow}.

  For the measurability, Proposition~\ref{l3} gives \eqref{eq:toshow0} since $\mathcal{H}_X$ is convergence determining. So, we have shown all assumptions in Theorem~\ref{T1} and we obtained a Feller process $X$ as a solution of the $G_X$-martingale problem. Moreover, there exists a \cadlag modification of $X$ by Proposition~\ref{p.loccomp-cadl}. Again, since $G_X$ is a second order operator, this solution has continuous paths by Proposition~4.5 in \cite{DGP12}.
\end{example}

\begin{example}[Interacting Fleming-Viot and Cannings]
  \label{ex.1249}
  Here we add \emph{space} and \emph{migration} to the Fleming-Viot process or the Cannings process from Examples~\ref{ex:FV} and \ref{ex.can-mod}. The approach is based on Remark~\ref{cor.814}. For some countable, discrete Ablian group $G$, we assume that $a: G\times G \to \mathbb R$ is a transition kernel such that $a(\xi,\eta) = a(0, \eta - \xi)$ with $\sum_\xi a(0,\xi) < \infty$. Setting $E_X = (\mathcal M_1(I))^{\mathbb G}$ and $E_Y = \bigcup_{m=0}^\infty G^{m} \times \mathcal C(I^{m})$, we use for $\underline\xi\in \mathbb G^{m}$ the probability measure $x_{\underline \xi} \coloneqq x_{\xi_1} \otimes \cdots \otimes x_{\xi_m}$ to define the duality function
  \begin{align}
    \label{eq:1002}
    H(x; \underline\xi,y) \coloneqq  \langle x_{\underline\xi},
    y\rangle \coloneqq  \int x_{\underline\xi}(du) y(u).
  \end{align}
  The migration operator reads
  \begin{align}
    \label{eq:136newb}
    \begin{split}
      G_X^{(1)}H(\cdot; \underline\xi, y)(x)
      & = \sum_{i=1}^{|y|} \sum_{\eta \in \mathbb G} a(\xi_i,\eta)
      \langle x_{\underline\xi_i^\eta} - x_{\underline\xi} , y\rangle,
    \end{split}
  \end{align}
  where $\underline\xi_i^\eta = (\xi_1,\dots,\xi_{i-1}, \eta, \xi_{i+1},\dots)$, whereas resampling and mutation operator are joined in
  \begin{align}
    \label{eq:136newc}
    \begin{split}
      G_X^{(2)}H(\cdot; \underline\xi, y)(x)
      & = \sum_{1\leq i<j\leq |y|} \ind{\xi_i = \xi_j}\langle
      x_{\underline\xi}, y\circ \theta_{ij} - y\rangle + \vartheta
      \sum_{1\leq k\le |y|} \langle x_{\underline\xi}, \beta_k y -
      y\rangle.
    \end{split}
  \end{align}
  The dual process $Y$ is a system of delayed coalescing random walks. More precisely, the $\underline\xi$-component moves according to the random walk kernel $a$, and the function-component $y$ follows \emph{coalescence} of coordinates $i$ and $j$ at \emph{unit rate}, if $\xi_i = \xi_j$, and changes to $\beta_k y$ at rate $\vartheta$. With $G_Y^{(i)}H(x,\cdot)(\xi,y) \coloneqq  G_X^{(i)}H(.;\underline\xi, y)(x)$, $i=1,2$, note that the dynamics of the $\underline\xi$-component happens according to $G_Y^{(1)}$, and of the function-component $y$ according to $G_Y^{(2)}$.

  We sketch the application of Remark~\ref{cor.814}, i.e.\ we need to check conditions (a), (b), (c) of Theorem~\ref{T572}: As in Example~\ref{ex:FVmut}, span$(\mathcal H_X)$ is a convergence determining algebra of functions containing~1, i.e.\ (a) holds. Then, the $(G_Y^{(i)}, \mathcal H_Y)$ martingale problems for $i=1,2$ are well-posed (for $G_Y^{(1)}$, we obtain a system of random walks, for $G_Y^{(2)}$, we obtain independent coalescence processes), as well as the $(G_Y \coloneqq G_Y^{(1)} + G_Y^{(2)},\mathcal H_Y)$ martingale problem, its solution being a Markov jump process, the delayed spatial coalescent with random walk  kernel $a$. Since both dual processes have bounded rates, they are Feller and (iii) of Proposition~\ref{l.480} holds. For (iv) of Proposition~\ref{l.480}, we have to show that the (well-defined) \emph{mass flow} induced by the transition kernel $a$ as given through $G_X^{(1)}$ is in duality with independent random walks on $\mathbb{G}$. This is well-known; cf.\ \cite{DGV95}, and uses that $G^{(1)}_X$ is a first order operator; see Remark~\ref{rem:dsomp}. In particular, this gives well-posedness of the $(G_X^{(1)}, \mathcal H_X)$ martingale problem and therefore existence of $\mu_t^{(1)}$, as indicated in Remark~\ref{cor.814}. For $G_X^{(2)}$, we see that a $\mathbb G$-indexed family of processes, distributed independently as solutions of the $(G_X, \mathcal H_X)$-martingale problem from Example~\ref{ex:FVmut}, is the unique solution of the $(G^{(2)}_X \mathcal H_X)$ martingale problem. Altogether, (b) of Theorem~\ref{T572} holds. For (c), we note that $Y$ (with generator $G_Y^{(1)} + G_Y^{(2)}$ again is a Markov jump process with bounded jump rates, hence Feller. So, Theorem~\ref{T572} gives well-posedness of the $(G_X, \mathcal H_X)$ martingale problem, and a modification with \cadlag paths. Again, since $G_X$ is a second order operator,  this solution has continuous paths by Proposition~4.5 in \cite{DGP12}.
\end{example}

\begin{remark}[Extensions of Example~\ref{ex.1249}]
  The above construction reproduces in particular the existence result
  of the system of \emph{interacting Fleming-Viot processes} on the
  discrete hierarchical group, discussed in
  \cite{EvansFleischmann1996} and gives existence on general discrete
  abelian groups in an alternative way to \cite{DGV95}.

  In order to obtain the spatial Cannings model from \cite{GHKK14}, we
  would then only need to check that Theorem~\ref{T1} is applicable to
  the non-spatial Cannings process, since the operator in that case is
  an integral operator and not a second order differential operator,
  the argument in Example~\ref{ex:FV} has to be adapted at the point
  of the calculation done in~\eqref{eq:139}, which are based on the
  property~\eqref{eq:7} of the dual dynamics and has been detailed in
  Example~\ref{ex.can-mod} for the non-spatial Cannings model. Hence
  this provides the details for the argument in \cite{GHKK14}.
\end{remark}

\begin{example}[Spatial Lambda-Fleming-Viot process]
\label{ex:slvf}
  In a series of papers, Etheridge and co-authors have introduced and studied the spatial $\Lambda$-Fleming-Viot process \cite{BEV10,EtheridgeFreemanStraulino2017,    EtheridgeKurtz2019,EtheridgeVeberYu2020}. The model has been studied in more detail and extended in \cite{VeberWakolbinger2015,  chetwynd2019rare, louvet2022growth}. In early papers, the existence of the process via the solution of a martingale problem is actually obtained using duality via the same approach as in \cite{Evans1997}; see \cite{BEV10}. Later papers show existence by using convergence of approximate models; see e.g.\ \cite{EtheridgeVeberYu2020}. Here, we will describe how existence of a solution of the martingale problem in its most basic form can be constructed using duality as an application of Corollary~\ref{cor.480}. Here, we will use a novel function-valued dual. Dual processes similar in spirit have been used, but they were not function-valued. 
  
  We will use the spatial $\Lambda$-Fleming-Viot model with type space $\{0,1\}$ excluding mutation and selection.

  Fix $d \in \mathbb N$, set $\mathbb E := \mathbb R^d$, as well as the state space of the process $X$, which is
  \begin{align}
    \label{eq:40}
    E_X := \{x \text{ measure on } \mathcal B(\mathbb E \times
    \{0,1\}): \pi_\ast x = \lambda\},
  \end{align}
  where $\pi: \mathbb E \times \{0,1\} \to \mathbb E$ is a projection and $\lambda$ is Lebesgue-measure on $\mathbb E$. (Some $x\in E_X$ models a population with constant density across $\mathbb E$, but the density of types $\kappa \in \{0,1\}$ may vary.) Equipping $E_X$ with the vague topology makes it a compact metric space; see Lemma~1.1 in  \cite{VeberWakolbinger2015}. Following Theorem~3.4 of  \cite{Kallenberg2021}, recall that for each $x\in E_X$, there is a Markov kernel from $\mathbb E$ to $\{0,1\}$ such that 
  \begin{align}
      \label{eq:w}
      \int f(u, \kappa) x(du, d\kappa) = \int \lambda(du) \int w(u, d\kappa) f(u, \kappa) \text{ for all }f\in\mathcal C_c(\mathbb E \times \{0,1\}).  
  \end{align}
  We start with an informal description of the process. Let $\mu$ be a $\sigma$-finite measure on $(0,\infty) \times (0,1]$ such that
  \begin{align}
    \label{eq:spatialcond}
    \int r^d p  \mu(dr, dp) < \infty,
  \end{align}
  and $\Pi$ be a Poisson process on $[0,\infty) \times \mathbb E\times (0,\infty) \times (0,1]$ with intensity measure $dt \otimes \lambda(dv) \otimes \mu(dr, dp)$. Then, for $(t,v,r,p) \in \Pi$, and if the current state of the process is $x$, set $K=\kappa$ with probability proportional to $x(B_r(v), d\kappa)$, $\kappa \in \{0,1\}$, and the process changes to
  \begin{align*}
    x_{v, r, p, K}(du, d\kappa)
    & := 1_{|u-v|\geq r} x(du, d\kappa) + 1_{|u-v|<r} \Big( (1-p)x(du,
      d\kappa) + p \lambda(du) \otimes \delta_K(d\kappa)\Big) \\
    & = \lambda(du) ( p1_{|u-v|<r} \delta_K(d\kappa) + (1 -
      p1_{|u-v|<r}) w(u, d\kappa)).
  \end{align*}
  This means that the offspring of one parent, chosen at random from $B_r(v)$ replaces a fraction $p$ of the total population within $B_r(v)$. The offsrpting inherits the type of their parents. Since $\lambda \otimes \mu$ is an infinite measure, there are infinitely many events within each time interval, and the issue of existence of such a process must be answered.

  For the function-valued dual process, let
  \begin{align}
    \label{eq:4432}
    E_Y = \bigcup_{n=1}^\infty \Pi_n, \qquad \Pi_n := \Big\{y\in
    \mathcal C_c(\mathbb E^{n}): y\geq 0, \int y(u) \lambda(du) =
    1\Big\}
  \end{align}
  and set $|y|=n$ if $y\in\Pi_n$, i.e.\ $\Pi_n$ consists of densities for samples of size $n$ with compact support. For the duality function, we use $y\in E_Y$ as well as the duality function (recall $w$ from \eqref{eq:w} and set $u:=(u_1,\dots,u_{|y|})$)
  \begin{align}
    \label{eq:spat_dual}
    H(x,y)
    & = \int \bigotimes_{i=1}^{|y|} \lambda(du_i) y(u) \int
      \bigotimes_{i=1}^{|y|} w(u_i, d\kappa_i) \kappa_1\cdots
      \kappa_{|y|}.
  \end{align}
  In words, this gives the probability of choosing a sample of $|y|$ individuals of type~1, if they are sampled according to the density $y$. On such functions, we are going to show existence of a process solving the martingale problem for the operator
  \begin{equation}
    \label{eq:spat_GX}
    \begin{aligned}
      G_X H(.,y)(x)
      & = \int \lambda(dv) \int  \mu(dr, dp) \frac{1}{|B_r(0)|} \int
        \lambda(du') 1_{|v-u'|<r} \int w(u', d\kappa') \\
      & \qquad \qquad \qquad  \cdot \int (x^{\otimes |y|}_{v,
        r, p, \kappa'}(du, d\kappa) - x^{\otimes |y|}(du,
        d\kappa))\kappa_1 \cdots \kappa_{|y|} y(u) .
    \end{aligned}
  \end{equation}
  In order to evaluate the right hand side and find a function-valued dual, observe that
  \begin{equation}
  \label{eq:spat_prod}
      \begin{aligned}
      x^{\otimes n}_{v, r, p, \kappa'}
        & (du, d\kappa) - x^{\otimes n}(du, d\kappa) \\
        & = \sum_{k=0}^n \sum_{\substack{I \subseteq \{1,\dots,n\}}{|I|=k}}\bigotimes_{i=1}^n \lambda(du_i) \Big(\prod_{i\in I}
          p1_{|u_i-v|<r} \Big( \bigotimes_{i\in I}
          \delta_{\kappa'}(d\kappa_i) - \bigotimes_{i\in I} w(u_i,
          d\kappa_i)\Big) \\
        & \qquad \qquad \qquad \qquad \qquad \qquad \qquad \qquad
          \cdot \prod_{i\notin I} (1-p1_{|u_i-v|<r}) \bigotimes_{i
          \notin I} w(u_i, d\kappa_i),
      \end{aligned}
  \end{equation}
  and we can write
  \begin{align}
  \label{eq:spat_1}
    G_X H(\cdot,y)(x)
    & = \sum_{k=1}^n
  \sum_{\substack{I\subseteq \{1,\dots,n\}}{|I| = k}}   A_{n,I} (x,y),
      \intertext{with $k\coloneqq |I|$ and some $\iota \in I$ (plugging
      \eqref{eq:spat_prod} in \eqref{eq:spat_GX} for
      \eqref{eq:spat_2}, evaluating the integrals with respect to
      $\kappa_i, i\notin I$ for \eqref{eq:spat_3}, and using $-\prod
      1_{|u_i-v|<r} = (1-\prod 1_{|u_i-v|<r}) - 1$ for
      \eqref{eq:spat_4})} \label{eq:spat_2} A_{n,I} (x,y)
    & = \int \lambda(dv) \int  \mu(dr, dp) p^k \frac{1}{|B_r(0)|} \int
      \lambda(du') 1_{|v-u'|<r} \int w(u',d\kappa') \\ \notag
    & \qquad \cdot \Big( \int \bigotimes_{i=1}^n \lambda(du_i)
      \prod_{i\in I} 1_{|u_i-v|<r} \Big( \int \bigotimes_{i\in
      I}\delta_{\kappa'}(d\kappa_i) \kappa_i - \int \bigotimes_{i\in
      I} w(u_i, d\kappa_i)\Big) \\ \notag
    & \qquad \qquad \qquad \qquad \qquad
      \cdot \int \prod_{i\notin I} (1 - p1_{|u_i-v|<r}) \int
      \bigotimes_{i \notin I} w(u_i, d\kappa_i) y(u) \kappa_1 \cdots
      \kappa_n \Big) \\
   & \label{eq:spat_3}
     = \int \mu(dr, dp) p^k |B_r(0)| \cdot \Big( \int \lambda(du')
     \int \bigotimes_{i\notin I} \lambda(du_i) \\
    \notag
    & \qquad \Big( \frac{1}{|B_r(0)|^2} \int  \lambda(dv) 1_{|v-u'|<r}
      \int \bigotimes_{i\in I} \lambda(du_i) \prod_{i\in I}
      1_{|u_i-v|<r} \prod_{i\notin I} (1 - p1_{|u_i-v|<r}) \Big) \\
    \notag
    & \quad  \qquad \cdot \Big( \int w(u', d\kappa') \kappa' - \int
      \bigotimes_{i\in I} w(du_i, d\kappa_i) \prod_{i\in I}
      \kappa_i\Big)  \int \bigotimes_{i\notin I} w(u_i, d\kappa_i)
      y(u) \prod_{i\notin I} \kappa_i \Big) \\
    & \label{eq:spat_4} = \int \mu(dr, dp) p^k |B_r(0)| \cdot \Big(
      \int \lambda(du') \int \bigotimes_{i\notin I} \lambda(du_i) \\
    \notag
    & \qquad \Big( \frac{1}{|B_r(0)|^2} \int  \lambda(dv) 1_{|v-u'|<r}
      \int \bigotimes_{i\in I} \lambda(du_i) y(u)  \prod_{i\in I}
      1_{|u_i-v|<r} \prod_{i\notin I} (1 - p1_{|u_i-v|<r}) \Big) \\
    \notag
    & \qquad \qquad \qquad \qquad \qquad \qquad \qquad \qquad
      \cdot\int w(u', d\kappa') \kappa'  \int \bigotimes_{i\notin I}
      w(u_i, d\kappa_i) \prod_{i\notin I}\kappa_i \\
    \notag
    & \qquad + \Big( \int \bigotimes_{i = 1}^n \lambda
      (du_i)\frac{1}{|B_r(0)|^2} \int \lambda(du') \int  \lambda(dv)
      \int 1_{|v-u'|<r} 1_{|u_\iota-v|<r} y(u) \\
    \notag
    & \qquad \qquad \qquad \cdot \Big(1 - \prod_{i\in I\setminus
      \{\iota\}} 1_{|u_i-v| < r} \prod_{i\notin I} (1 -
      p1_{|u_i-v|<r}) \Big)\Big) \int \bigotimes_{i=1}^n w(u_i,
      d\kappa_i) \prod_{i=1}^n \kappa_i \\
    \notag
    & \qquad \qquad \qquad \qquad \qquad \qquad -
      \int \bigotimes_{i=1}^n \lambda(du_i) y(u) \int
      \bigotimes_{i=1}^n w(u_i, d\kappa_i) \prod_{i=1}^n
      \kappa_i\Big).
  \end{align}
  We interpret the right hand side saying that $y$ jumps to $y'_{I,r,p} + y''_{I,r,p}$ with $y'_{I,r,p} \in \Pi_{n-k+1}$ (note that $1\leq n-k+1\leq n$ since $1\leq |I|=k\leq n$) and $y''_{I,r,p} \in \Pi_n$ at rate $\mu(dr, dp) p^k(1-p)^{n-k} |B_r(0)|$, where
  \begin{align}
    \label{eq:spat_5a}
    y'_{I,r,p}(u', (u_i)_{i\notin I})
    & = \frac{1}{|B_r(0)|^2} \int  \lambda(dv) 1_{|v-u'|<r} \int
      \bigotimes_{i\in I} \lambda(du_i) y(u)  \prod_{i\in I}
      1_{|u_i-v|<r} \prod_{i\notin I} (1 - p1_{|u_i-v|<r}) \\
    \label{eq:spat_5b} y''_{I,r,p}(u)
    & = \frac{1}{|B_r(0)|^2} \int \lambda(du') \int  \lambda(dv)
      1_{|v-u'|<r} 1_{|u_\iota-v|<r}  y(u) \\
    \notag
    & \qquad \qquad \qquad \qquad \qquad \qquad
      \cdot \Big(1 - \prod_{i\in I\setminus \{\iota\}} 1_{|u_i-v| < r}
      \prod_{i\notin I} (1 - p1_{|u_i-v|<r}) \Big).
  \end{align}
  Note that by construction 
  \begin{align}
    \int \lambda(du') \bigotimes_{i\notin I} \lambda(du_i)
    y'_{I,r,p}(u', (u_i)_{i\notin I}) + \int \bigotimes_{i=1}^n
    \lambda(du_i) y''_{I,r,p}(u) = 1 
  \end{align}
  and we can change variables in $y'_{I,r,p}$ (i.e.\ changing $u'$ with $u_\iota$) such that $y'_{I,r,p} + y''_{I,r,p}$ depend on the same variables $u_1,\dots,u_n$.

  Let us use these calculations in order to show existence and uniqueness of the $G_X$-martingale problem. We will use Theorem~\ref{T1}, in particular Corollary~\ref{cor.480}, using the dual process $Y = (Y_t)_{t\geq 0}$. We argue as follows: 
  First, $\text{span}(\mathcal H_X)$ is a convergence determining algebra since $x_n \to x$ vaguely if and only if $\int y(u) x_n(du) \xrightarrow{n\to\infty} \int y(u) x(du)$ for all $y \in \mathbb  C_c(\mathbb E)$; see also Lemma~1.1 in \cite{EtheridgeVeberYu2020}. Second, note that for $Y_t = y \in \Pi_j$, the dual process jumps to $y'_{I,r,p}+y''_{I,r,p}\in\Pi_n$ at rate $\int \mu(dr, dp) p^k |B_r(0)| < \infty$, as we seen above. Note that this rate is bounded by the left hand side of \eqref{eq:spatialcond}, such that $Y$ is the unique solution of its martingale problem and \eqref{eq:dual3new} holds (with $\beta=0$) for the function-valued pure jump process $Y$, which is $\mathcal C_b(E_Y)$-Feller. So, \eqref{eq:toshow0} and \eqref{eq:toshow} follow from Propositions~\ref{l3} and~\ref{l.480} provided we can show Proposition~\ref{l.480}(iv). For this, we use Corollary~\ref{cor.480} and argue similarly as in the proof of Theorem~1.2 of \cite{EtheridgeVeberYu2020}.   We use an approximating sequence of models $X^1, X^2,...$ with duals $Y^1, Y^2,...$, which arise by restricting for $X^n$ to reproduction events on $(-n,n)^d$ and some finite $\mu^n\leq \mu$ on $(0,\infty) \times (0,1]$, such that $\mu^n \uparrow \mu$ as $n\to\infty$. For these dual pairs $(X^n, Y^n)$, the construction guarantees:
  \begin{enumerate}
      \item[(i)] The dual processes converge, i.e.\ $Y^n \xRightarrow{n\to\infty} Y$, since $Y$ is a pure Markov jump process (with finite jump rate), and jumping intensities converge;
      \item[(ii)] The martingale problems for $X^1, X^2,...$ are well-posed and unique solutions of the corresponding martingale problems, since $X^n$ is a pure Markov jump process with finite jump rate.
  \end{enumerate}
  From Corollary~\ref{cor.480}, we see that Proposition~\ref{l.480}(iv) holds and thus, we have shown existence and uniqueness of the $(G_X, \mathcal H_X)$-martingale problem. 
\end{example}

\subsection{Locally-compact state spaces -- branching systems}
\begin{example}[Continuous state branching processes]
  \label{ex:CSBP}
  \sloppy For the construction of a superprocess, E.~Dynkin uses in \cite{Dynkin1993} what he calls the direct construction, which can be viewed as a duality argument. In fact, this approach is connected to Theorem~\ref{T1} which we demonstrate now for simplicity for a  \emph{non}-spatial branching system.

  The state space of the process that we wish to construct is $E_X=\R_+$. To define the operator let $b \in \mathbb R$, $c\in \mathbb R_+$ and let $N$ be a measure on $[0,\infty)$ with $\int_0^\infty (s\wedge s^2) N(ds) < \infty$ and $\int_{\,0+} s^2 N(ds) = 0$. We set
  \begin{align}
    \label{eq:21}
    \mathcal D_X = C_c^2(\R_+),
  \end{align}
  where $C_c^2(\R_+)$ denotes the set of twice continuously
  differentiable real-valued functions on $\R_+$ with compact support. The operator for the process we aim to construct is given by (see \cite{DawsonLi2006} eq. (5.23) for a more general case)
  \begin{align}
    \label{eq:20}
    G_X f(x) = b x f'(x) + c x f''(x) + x \int_0^\infty (f(x+s) - f(x)
    - s f'(x))\, N(ds).
  \end{align}
  Note that for $N=0$, this is the generator of a Feller diffusion with drift. Let $H: \R_+ \times \R_+$ with $H(x,y)=e^{-xy}$ and let $Y^y$ be the deterministic process satisfying $Y_0=y$ and solving
  \begin{align}
    \label{eq:laplace}
    \dot Y = - \Psi(Y) \text{ with } \Psi(y) = by + cy^2 +
    \int_0^\infty(e^{-sy} - 1 + sy)N(ds).
  \end{align}
  Here, $\Psi$ is usually referred to as the \emph{branching mechanism}. The generator of $Y$ is given by
  \begin{align}
    \label{eq:13}
    \begin{aligned}
      G_Y e^{-x.}(y) & = -\Psi(y) \frac{\partial}{\partial y} e^{-xy}
      = x\Psi(y) e^{-xy}
      = \Big(by + cy^2 + \int_0^\infty(e^{-sy}-1+sy)N(ds)\Big)x e^{-xy} \\
      & = b x \frac{\partial}{\partial x} e^{-xy} + c x
      \frac{\partial^2}{\partial x^2} e^{-xy} + x\int_0^\infty
      \Big(e^{-(s+x)y} - e^{-xy} - s \frac{\partial}{\partial x}
      e^{-xy}\Big) N(ds).
    \end{aligned}
  \end{align}
  Then, for \eqref{eq:toshow}, we need to find a random variable $X_t$ such that, for $Y_t^y$ solving \eqref{eq:laplace} with $Y_0=y$,
  \begin{align}
    \label{eq:14}
    \mathbb E_x[e^{-yX_t}] = e^{-xY^y_t} \eqqcolon \psi(y).
  \end{align}
  So, we need to see if $\psi(y)$ is the Laplace transform of some $\R_+$-valued random variable. This is equivalent to the following four conditions: (i) $\psi$ is continuous, (ii) $\psi$ is positive definite, (iii) $\psi\geq 0$ and (iv) $\psi(0)=1$. See for instance  \cite[Corollary~4.5, p.~114]{Berg1984} for the case of finite measures and note that (iv) ensures that we have a probability measure. Clearly, (i), (iii) and (iv) are satisfied. Condition (ii) is equivalent to the requirement that $y\mapsto Y_t^y$ is negative definite; see \cite[Proposition~6.10, p.~133]{Berg1984}. This,  however, is proved in \cite[Proposition 3.2(v)]{Beznea2011}, and hence, we have shown \eqref{eq:toshow}. Finally, \eqref{eq:toshow0} follows as in Proposition~\ref{l3}. Since $Y$ is Feller, $X$ is  Feller as well. For path regularity, the compact containment condition for $X$ can be proved using a priori moment bounds to get compact containment for fixed times $t$ and then using Doob's inequality to for the argument on the paths space. Then the existence of a \cadlag modification follows;  compare with Remark~\ref{rem:R1}.
\end{example}

\noindent
Now we give an example with $\beta\neq 0$ in the duality relation \eqref{eq:toshow}. We prepare this example with two lemmas.

\begin{lemma}
  \label{L:jp}
  Let $Y = (Y_t)_{t\geq 0}$ be a pure jump process with countable state space and denote by $y_0$ the start point of $Y$ and by $Y_k$ the state of $Y$ after the $k$th jump, $k=1,2,\dots$. Moreover, the total jump rate of $Y$ in state $y'$ is denoted $\gamma(y')$ and the jump rate from $y'$ to $y''$ by $\gamma(y'\to y'')$. Then, for any $f$,
  \begin{align}
    \mathbb E\Big[ f(Y_t) \cdot \exp\Big( \int_0^t \gamma(Y_s) ds\Big)\Big]
    = \sum_{n=0}^\infty \frac{t^n}{n!}  \sum_{y_1, \dots, y_n}
    f(y_n) \prod_{k=0}^{n-1} \gamma(y_k\to y_{k+1}),
  \end{align}
  where $\prod_{k=0}^{-1}\coloneqq 1$, if the right hand side exists.
\end{lemma}

\begin{proof}
  Let $N_t$ be the number of jumps before time $t$. Then for $n\geq 1$ we can compute as follows
  \begin{align*}
    \mathbb E\Big[
    & f(Y_t) \cdot \exp\Big( \int_0^t \gamma(y_s) ds\Big), N_t=n\Big] \\
    & = \sum_{y_1,\dots,y_n} \int_0^t dt_1 \gamma(y_0)
      \frac{\gamma(y_0\to y_1)}{\gamma(y_0)}e^{-\gamma(y_0) t_1}
      \int_{t_1}^t dt_2 \gamma(y_1) \frac{\gamma(y_1\to y_2)}{\gamma(y_1)}
      e^{-\gamma(y_1) (t_2-t_1)} \\
    & \qquad \qquad \qquad \qquad \cdots \int_{t_{n-1}}^t dt_n
      \gamma(y_{{n-1}}) \frac{\gamma(y_{n-1}\to y_n)}{\gamma(y_{n-1})}
      e^{-\gamma(y_{{n-1}})(t-t_{n-1})} \cdot e^{-\gamma(y_n)(t-t_n)} \\
    & \qquad \qquad \qquad \qquad \qquad \qquad
      f(y_n) e^{t_1 \gamma(y_0) + (t_2-t_1) \gamma(y_1) + \dots +
      \gamma(y_n)(t-t_{n})} \\
    & = f(y_n) \prod_{k=1}^{n} \gamma(y_{k-1}\to y_{k}) \cdot
      \int_0^t dt_1 \int_{t-t_1}^t dt_2 \cdots \int_{t-t_{n-1}}^t dt_n \\
    & = \frac{t^n}{n!} \sum_{y_1, \dots, y_n} f(y_n)
      \prod_{k=0}^{n-1} \gamma(y_k\to y_{k+1}).
  \end{align*}
  An analogous equation holds for $n=0$. Summing over $n$ gives the assertion.
\end{proof}

The following result is standard and formulated here for reference in the next example.

\begin{lemma}[Moments, Bernstein functions and Laplace transforms]
  \label{lem:bernstein-laplace}
  Let $(\mathfrak{m}_y)_{y=0,1,\dots}$ be a sequence of non-negative real numbers. Define $\psi: (0,\infty) \to \R$ by
  \begin{align}
    \label{eq:psi-lapl}
    \psi(\lambda) = \sum_{y=0}^\infty \frac{(-\lambda)^y}{y!} \mathfrak{m}_y.
  \end{align}
  Assume that for some $x > 0$ there is a function $\varphi$ so that $\psi(\lambda) = e^{-x \varphi(\lambda)}$ for all $\lambda>0$. If $\varphi$ admits the representation
  \begin{align}
    \label{eq:bernstein}
    \varphi(\lambda) = a + b\lambda + \int_{(0,\infty)} (1-e^{-\lambda r})\,
    \nu(dr),
  \end{align}
  where $a,b \ge 0$ and $\nu$ is a measure on $(0,\infty)$ satisfying $\int_{(0,\infty)} (1\wedge r ) \, \nu(dr)<\infty$, then there exists a unique non-negative measure $\mu$ on $[0,\infty)$ so that 
  \begin{align*}
    \psi(\lambda) = \int_{[0,\infty)} e^{-\lambda r} \, \mu (dr).
  \end{align*}
\end{lemma}

\begin{proof}
  The assertion of the lemma follows by a combination of results from \cite{SchillingSongVondracek2012}. By  \cite[Theorem~3.2]{SchillingSongVondracek2012} the function $\varphi$ is a Bernstein function which by  \cite[Theorem~3.7]{SchillingSongVondracek2012} is equivalent to the fact that $\psi$ is a completely monotone function. By \cite[Theorem~1.4]{SchillingSongVondracek2012} it must be a Laplace  transform of a unique measure $\mu$ on $[0,\infty)$.
\end{proof}

\begin{example}[Feller's branching diffusion] \label{ex:Feller}
  Here, we have $E_X = \R_+$. The operator and its domain are given by
  \begin{align}
    \label{eq:22}
    G_X f(x) = \frac{1}{2} \frac{\partial^2}{\partial x^2} f(x), \qquad
    \mathcal D_X=C^2_b ([0,\infty)).
  \end{align}
  The state space of the dual process is $E_Y = \mathbb N$. For the duality function we choose -- similar to Example~\ref{ex:FV} -- $H(x,y) = x^y$ and we let $Y$ be the Markov jump process with generator
  \begin{align}
    \label{eq:1}
    G_Yf(y) = \binom y2 (f(y-1) - f(y)).
  \end{align}
  For $\beta(y) = \binom y2$, this gives
  \begin{align}
    \label{eq:2}
    G_XH(\cdot,y)(x)
    & \coloneqq G_Y H(x,\cdot)(y) + \beta(y) H(x,y)
      = \binom y2 x^{y-1} = \frac 12 x \frac{\partial^2}{\partial x^2} x^y,
  \end{align}
  which we recognize as the generator of Fellers's branching diffusion on $[0,\infty)$. Hence, for $x\in E_X$ and $t \ge 0$, in order to show \eqref{eq:toshow}, we need to find (the law of a random variable) $X_t$ such that for all $y \in E_Y$ we have
  \begin{align}
    \label{eq:3}
    \mathbb E_x[X_t^y]
    = \mathfrak{m}_y \coloneqq
    \mathbb E_y\Big[x^{Y_t} \exp\Big(\int_0^t \binom{Y_s}{2}\, ds\Big)\Big]
    = \sum_{n=0}^{y-1} \frac{t^n}{n!}x^{y-n} \prod_{k=0}^{n-1} \binom{y-k}2,
  \end{align}
  where we have used Lemma~\ref{L:jp} in the last step. (Note that the product of binomial coefficients is interpreted as $1$ in cases $y=1$ or $n=0$.) In order to find $X_t$, we will use Lemma~\ref{lem:bernstein-laplace}. Setting $\mathfrak{m}_0=1$ we have a sequence $(\mathfrak{m}_y)_{y=0,1,\dots}$ and for $\psi$ as in \eqref{eq:psi-lapl} we obtain
  \begin{align}
    \label{eq:4}
    \begin{split}
      \psi(\lambda)
      & = \sum_{y=0}^\infty \frac{(-\lambda)^y}{y!}
      \mathfrak{m}_y = 1 + \sum_{y=1}^\infty\sum_{n=0}^{y-1}
      \frac{(-\lambda)^y}{y!} \frac{t^n}{n!}x^{y-n} \,
      \prod_{k=0}^{n-1} \binom{y-k}2 \\
      & = 1 +\sum_{y=1}^\infty\sum_{n=0}^{y-1} \frac{(-\lambda x)^y}{n!}
      \Bigl(\frac{t}{2x}\Bigr)^n \, \frac{(y-1)!}{(y-n-1)!(y-n)!}\\
      & = 1 + \sum_{y=1}^\infty\sum_{n=0}^{y-1}
      \frac{(-\lambda x)^y}{(y-n)!}
      \Bigl(\frac{t}{2x}\Bigr)^n \binom{y-1}{y-n-1} = 1
      + \sum_{y=1}^\infty\sum_{n=1}^{y} \frac{(-\lambda x)^y}{n!}
      \Bigl(\frac{t}{2x}\Bigr)^{y-n} \binom{y-1}{n-1}\\
      & = 1 + \sum_{y=1}^\infty\sum_{n=1}^{y} \frac{(-\lambda t/2)^y}{n!}
      \Bigl(\frac{2x}{t}\Bigr)^{n} \binom{y-1}{n-1} = 1
      + \sum_{n=1}^{\infty} \frac{(2x/t)^n}{n!}
      \sum_{y=n}^\infty \Bigl(-\frac{\lambda t}{2}\Bigr)^y \binom{y-1}{n-1} \\
      & = 1 + \sum_{n=1}^{\infty} \frac{(2x/t)^n}{n!} \Bigl(-
      \frac{\lambda t/2}{1+\lambda t/2}\Bigr)^n = \sum_{n=0}^{\infty}
      \frac{1}{n!} \Bigl(-\frac{\lambda x}{1+\lambda t/2}\Bigr)^{n}
      = \exp\Bigl(-\frac{\lambda x}{1+\lambda t/2}\Bigr).
    \end{split}
  \end{align}
  Apart from several elementary manipulations we have used that $\sum_{k=i}^\infty \, (-a)^k \binom{k}{i} =  \frac{(-a)^i}{(1+a)^{i+1}}$ in the first equality of the last line. Now the function $\varphi(\lambda) =\frac{\lambda }{1+\lambda t/2}$ can be written in the form \eqref{eq:bernstein} with $a=b=0$ and $\nu(dr) = (t/2)^{-2} \exp(-r/(t/2))\, dr$. Indeed we have
  \begin{align}
    \label{eq:12}
    \int_{(0,\infty)} (1-e^{-\lambda r}) (t/2)^{-2} e^{-r/(t/2)} \,dr
    = \frac{4}{t^2} \Bigl(t/2 - \frac{t/2}{1+\lambda t/2}\Bigr)
    = \varphi(\lambda).
  \end{align}
  Now the existence of $X_t$ or more precisely the existence and uniqueness of the corresponding laws follows by Lemma~\ref{lem:bernstein-laplace} and we have obtained $(\mu_t)_{t\geq 0}$ as required in \eqref{eq:toshow0} and \eqref{eq:toshow}. We note that, Theorem~\ref{T1} cannot be applied directly since $H$ is unbounded. However, its generalization discussed in Remark~\ref{rem:genH} does apply and we obtain a Feller  process $(X_t)_{t\geq 0}$, which is the unique solution to the $(G_X,{\mathcal H}_X,x)$-martingale problem for all $x$ and setting $P^\mu=\int_{(0,\infty)} P^x \mu(dx)$, with $P^x$ the solution starting in $x$ we get the unique solution to the $(G_X,{\mathcal H}_X,\mu)$-martingale problem for all $\mu \in \mathcal{M}_1(\mathbb{R}^+)$. Since $X$ is a martingale, it has a \cadlag modification, and therefore continuous paths since $G_X$ is a second order operator.
\end{example}

\section{Outlook on non-locally compact state spaces}
\label{ss.outcomp}
Several of our results require $E_X$ to be compact. In particular, Proposition~\ref{l.480} is based on an application of the Riesz-Markov theorem, which works best for compact spaces. Also the proof of Theorem~\ref{T572} uses Proposition~\ref{l.480}. Hence in the cases of non-compact Polish state spaces we need to work with a suitable \emph{compactification}.

Typical examples of state spaces $E_X$ which are \emph{not} locally compact arise in models involving a continuum spatial component, genealogies or some function spaces. Examples where existence by duality was already obtained in the literature are \cite{Evans1997}, \cite{Dynkin1993} and \cite{BEV10}. We briefly discuss (i) \emph{historical processes} and (ii) \emph{genealogy-valued} processes.

Let $\mathbb G$ be a countably infinite abelian group. In (i), the state space is $\mathcal M(\mathcal D(\mathbb{R},\mathbb{G}))$, where $\mathcal D(\mathbb{R},\mathbb{G})$ is the set of \cadlag paths on $\mathbb G$; see \cite{DP91,D93}. The idea is to associate with every individual alive at time $t$ its \emph{path of descend} describing the geographical position of its ancestor at times $s \in [0,t]$ and extend the path before time $0$ and after time $t$ as a constant path. Then the state space is a (locally finite) measure on the set of such paths and hence we have in general a non-locally compact state space.

For (ii), the state space is called $\mathbb U^{\mathbb G}$, which is the set of (equivalence classes of) $\mathbb G$-marked metric measure spaces, i.e.\ triples $(X, r, \mu)$, where $(X,r)$ is a metric space (coding for the genealogy) and $\mu \in \mathcal M(X \times \mathbb G)$; see~\cite{GPW09,DGP12,GPWmp13,GSW,ggr_tvF14}. This leads to state spaces which are not $\sigma$-compact and not locally compact. In particular one needs to check tightness conditions to study convergence and path properties of stochastic processes.

In both cases, Theorem~\ref{T1} is applicable, but checking \eqref{eq:toshow0} and \eqref{eq:toshow} requires some additional work due to non-compactness of the state space. 
We note, that in studying such processes the technique of duality is very useful and applicable for our existence problem. We shall formulate below a criterion and a condition we need to verify in order to obtain the \emph{existence} of a solution. To check this condition one needs to develop methods to verify that the paths of the process in the compactified state spce remain in some subset whose preimage w.r.t.\ the embedding of the original space is contained in the original space itself. For fixed times $t$ this is known for genealogy-valued Fleming-Viot or Cannings models due to the so called strong duality. For all $t$, i.e.\ on the process level ongoing work in \cite{GKW23} suggests that this issue will be resolved in the context of genealogy process by the construction of the ancestral web and its dual.

Let us now discuss the announced approach useful for dealing with general \emph{Polish state spaces} $E_X$. The key is the following result, which reformulates and combines the strategies appearing in the literature e.g.\ in \cite{KSt01}. Note that we will be using this result for $\wh E$ compact.

\begin{proposition}[How to treat general state spaces]\label{l:490}
  Let $E$, $\wh E$ be Polish, $G: \mathcal D \subseteq \mathcal B(E) \to \mathcal B(E)$ and $\mathbb P_0 \in \mathcal M_1(E)$. Assume that $\Psi: E\to\wh E$ is injective and bi-measurable (i.e.\ $\Psi$ and    $\Psi^{-1}$ are measurable). Set
  \begin{align}
    \label{eq:comp0}
    \wh{\mathcal D}
    & \coloneqq \{\wh g_f \in \mathcal B(\wh E): f\in\mathcal D, \wh
      g_f|_{\Psi(E)} = f\circ \Psi^{-1}\},\\
    \label{eq:comp1}
    \wh G \wh g_f(\hat x)
    & \coloneqq
      \begin{cases}
        Gf(\Psi^{-1}(\hat x)), & \text{ if }\hat x \in \Psi(E), \\
        0, & \text{ otherwise,}
      \end{cases} \\
    \label{eq:comp2}
    \wh{\mathbb P}_0 & \coloneqq  \Psi_\ast \mathbb P_0.
  \end{align}
  \begin{enumerate}[(a)]
  \item If $X$ solves the $(G, \mathcal D, \mathbb P_0)$ martingale problem,
  then $\Psi(X)$ solves the $(\wh G, \wh{\mathcal D}, \wh{\mathbb
    P}_0)$ martingale problem. If, in addition, $\Psi$ is continuous and $X$ has \cadlag (continuous) paths, then $\Psi(X)$ has \cadlag (continuous) paths as well.
    \item If $\wh X$ solves the $(\wh G,  \wh{\mathcal D}, \wh{\mathbb P}_0)$ martingale problem and has paths in $\Psi(E)$, then $\Psi^{-1}(\wh X)$ solves the $(G, \mathcal D, \mathbb P_0)$ martingale problem. If, in addition, $\Psi^{-1}$ is continuous and $\wh X$ has \cadlag (continuous) paths, then $\Psi^{-1}(\wh X)$ has \cadlag (continuous) paths as well.
\end{enumerate}
\end{proposition}

Note that (b) can be used in various ways following literature to develop criteria which additionally have to be checked for the existence of solutions of the martingale problem.
Indeed if $E$ is locally compact, the above construction is well-known. In this case one can use the one-point compactification $\wh E \coloneqq E \cup \{\ast\}$ via $\Psi=$id; see e.g.\ Section~4.3 of \cite{EK86}.

An example for $\wh E$ in the case of \emph{not locally compact Polish space $E$} is as follows (see e.g.\ \cite[Section~3]{KSt01} and \cite{BhattKara93}): Assume that there is $\mathcal D' \subseteq \mathcal{D}$ countable and separating such that
\begin{align}
  \label{eq:10}
  \text{bp-closure of } \{(g, Gg): g\in\mathcal D'\} \supseteq
  \{(f, Gf): f\in\mathcal D\}.
\end{align}
Then consider the compact (in the product topology on $\mathbb R^{\mathbb N}$) set
\begin{align}
  \label{e463}
  \widehat{E}= \bigtimes_{g \in \mathcal D'} \bigl[-\sup \abs{g}, + \sup
  \abs{g}\bigr],
\end{align}
and use $\Psi: E_X \to \widehat{E}_X$ via
\begin{align}
  \label{e461}
  \Psi(x) = (g(x))_{g \in \mathcal D'}.
\end{align}
Since $\mathcal D'$ is separating, $\Psi$ is injective. If $\mathcal D' \subseteq \mathcal C_b(E)$, $\Psi$ is continuous, and if $\mathcal D'$ is convergence determining then $\Psi^{-1}$ is continuous (on $\Psi(E)$).

The question is now how to work with $\widehat E$. Suppose we want to use Proposition~\ref{l.480} or Theorem~\ref{T572}, writing $E_X$ and $\wh E_X$ for the state space of the process $X$. In case we want to use one of these results for showing \eqref{eq:toshow}, we can make use of Proposition~\ref{l:490} (assuming $\wh E_X$ is compact) and $\Psi: E_X \to\wh E_X$ is as in Proposition~\ref{l:490} as follows: We use $\mathcal{D}' \subset \{H(\cdot,y) : y \in E_Y\}$. Then we can extend the duality w.r.t\ to function $H$ to a duality w.r.t. function $\wh H$ on $\wh E_X$, $\wh H: \wh E_X \times E_Y \to \mathbb R$, satisfying (i), (ii) of Theorem~\ref{T1} and
\begin{align}
  \label{eq:tosq}
  \wh H(\Psi(x), y) = H(x,y), \qquad x\in E_X, y\in E_Y.
\end{align}
Then there is a family of transition kernels $(\wh \mu_t)_{t\geq 0}$ from $\wh E_X$ to $\wh E_X$ satisfying \eqref{eq:toshow0} and\eqref{eq:toshow} (which can be shown using e.g.\ Proposition~\ref{l.480} or Theorem~\ref{T572} due to compactness of
$\wh E_X$), if additionally we have the following containment property of a solution starting in $E_X$: 
\begin{align}
  \label{eq:tosh}
  \wh \mu_t(\Psi(x), \Psi(E_X))=1 \text{ for all $t\geq 0, x\in E_X$.}
\end{align}
We define $\mu_t(x,\cdot) = \Psi^{-1} _\ast \wh \mu_t(\Psi(x),.)$, for all $t\geq 0, x \in E_X$ and $y\in E_Y$. Then, \eqref{eq:toshow0} holds for $(\mu_t)_{t\geq 0}$, since $\Psi$ and $\Psi^{-1}$ are measurable (on $\Psi(E_X)$); see Section~3 of \cite{KSt01}. Moreover, by using successively \eqref{eq:tosq} and \eqref{eq:toshow} on $\wh E_X$, then the definition of $\mu_t$ and finally \eqref{eq:tosq} again we obtain
\begin{align}
  \label{eq:9}
  \begin{split}
    \mathbb E_{y}\Bigl[H(x,Y_t)
    \exp\Bigl(\int_0^t \beta(Y_s)\, ds\Bigr)\Bigr]
    & = \mathbb E_{y}\Bigl[\wh H(\Psi(x),Y_t)
      \exp\Bigl(\int_0^t \beta(Y_s)\, ds\Bigr)\Bigr] \\
    & = \int \wh H(x', y) \wh \mu_t(\Psi(x), dx') \\
    & = \int \wh H(\Psi(x'), y) \mu_t(x, dx') \\
    & = \int H(x', y) \mu_t(x, dx')
  \end{split}
\end{align}
which shows \eqref{eq:toshow} for $(\mu_t)_{t\geq 0}$. We obtain the following corollary.

\begin{corollary}
    Let $E_X, E_Y, H, G_Y, \beta, \mathcal H_X$ and $\mathcal H_Y$ be as in Theorem~\ref{T1}, and let Theorem~\ref{T1}(i) hold. In addition, let $E \coloneqq E_X$ and $\wh E$ be as in Proposition~\ref{l:490}. If $\wh X$ satisfies Proposition~\ref{l:490}(ii) (in particilar \eqref{eq:tosh}, then \eqref{eq:toshow0} and \eqref{eq:toshow} hold for $\wh X$, existence and uniqueness of a solution to the $G_X$-martingale problem follows.
\end{corollary}

As discussed in Remark~\ref{rem:R1} we need to \emph{check the regularity of paths separately}. Here it means that we have to check that \eqref{eq:tosh} holds as an additional condition and we use it in the dual process or an extension of it. The first step would be to establish \eqref{eq:tosh} for \emph{fixed} $t$ and to then in a second step exclude \emph{exceptional points} of the paths. This can sometimes be done using the dual process $Y$ or rather its extension to a strong duality. 

Then we find that the $(G_X, \mathcal H_X)$ martingale problem has a unique solution which has a \cadlag modification, since for general state spaces, Theorem~4.3.6 in \cite{EK86} states the existence of a \cadlag modification of the $(G_X,\mathcal H_X)$ martingale problem provided the compact containment condition holds.


\subsubsection*{Acknowledgements}
We thank Tom Kurtz for useful suggestions concerning the non-locally compact case and pointing out a reference. This research was supported by the DFG priority program SPP\,1590 through the grants Pf672/8-1 to PP and Gr876/16-1,2 to AG. AG was also supported by the DFG-Grant Gr876/17-1. PP is partially funded by the Freiburg Center for Data analysis and Modeling (FDM). 

\subsubsection*{Declaration of Generative AI and AI assisted technologies in the writing process}
During the preparation of this work the author(s) used no generative AI tools.


\begin{thebibliography}{GdHKK14}

\bibitem[BCR84]{Berg1984}
Christian Berg, Jens Peter~Reus Christensen, and Paul Ressel.
\newblock {\em Harmonic analysis on semigroups}, volume 100 of {\em Graduate
  Texts in Mathematics}.
\newblock Springer-Verlag, New York, 1984.
\newblock Theory of positive definite and related functions.

\bibitem[BEV10]{BEV10}
N.~Barton, A.~Etheridge, and A.~V\'eber.
\newblock A new model for evolution in a spatial continuum.
\newblock {\em Electron. J. Probab.}, 15:162--216, 2010.

\bibitem[BEV13]{Barton_2013}
N~H Barton, A~M Etheridge, and A~Véber.
\newblock Modelling evolution in a spatial continuum.
\newblock {\em Journal of Statistical Mechanics: Theory and Experiment},
  2013(01):P01002, jan 2013.

\bibitem[Bez11]{Beznea2011}
Lucian Beznea.
\newblock Potential-theoretical methods in the construction of measure-valued
  {M}arkov branching processes.
\newblock {\em J. Eur. Math. Soc. (JEMS)}, 13(3):685--707, 2011.

\bibitem[BK93]{BhattKara93}
A.G. Bhatt and R.L. Karandikar.
\newblock Invariant measures and evolution equations for {M}arkov processes.
\newblock {\em Ann. Probab.}, 21:2246--2268, 1993.

\bibitem[CDK19]{chetwynd2019rare}
Jonathan Chetwynd-Diggle and Aleksander Klimek.
\newblock Rare mutations in the spatial lambda-fleming-viot model in a
  fluctuating environment and superbrownian motion.
\newblock {\em arXiv preprint arXiv:1901.04374}, 2019.

\bibitem[Daw93]{D93}
Donald~A. Dawson.
\newblock Measure-valued {M}arkov processes.
\newblock In {\em \'{E}cole d'\'{E}t\'e de {P}robabilit\'es de {S}aint-{F}lour
  {XXI}---1991}, volume 1541 of {\em Lecture Notes in Math.}, pages 1--260.
  Springer, Berlin, 1993.

\bibitem[DG23]{ggr_tvF14}
Andrej Depperschmidt and Andreas Greven.
\newblock Genealogy-valued {F}eller diffusion.
\newblock {\em http://arxiv.org/abs/1904.02044}, submitted January 2023.

\bibitem[DGP12]{DGP12}
Andrej Depperschmidt, Andreas Greven, and Peter Pfaffelhuber.
\newblock Tree-valued {F}leming-{V}iot dynamics with mutation and selection.
\newblock {\em Ann. Appl. Probab.}, 22(6):2560--2615, 2012.

\bibitem[DGV95]{DGV95}
Donald~A. Dawson, Andreas Greven, and Jean Vaillancourt.
\newblock Equilibria and quasiequilibria for infinite collections of
  interacting {F}leming-{V}iot processes.
\newblock {\em Trans. Amer. Math. Soc.}, 347(7):2277--2360, 1995.

\bibitem[DL06]{DawsonLi2006}
D.~A. Dawson and Zenghu Li.
\newblock Skew convolution semigroups and affine {M}arkov processes.
\newblock {\em Ann. Probab.}, 34(3):1103--1142, 2006.

\bibitem[Doo53]{Doob1953}
J.~L. Doob.
\newblock {\em Stochastic processes}.
\newblock John Wiley \& Sons, Inc., New York; Chapman \& Hall, Limited, London,
  1953.

\bibitem[DP91]{DP91}
Donald~A. Dawson and Edwin~A. Perkins.
\newblock Historical processes.
\newblock {\em Mem. Amer. Math. Soc.}, 93(454):iv+179, 1991.

\bibitem[Dyn93]{Dynkin1993}
E.~B. Dynkin.
\newblock Superprocesses and partial differential equations.
\newblock {\em Ann. Probab.}, 21(3):1185--1262, 1993.

\bibitem[EF96]{EvansFleischmann1996}
Steven~N. Evans and Klaus Fleischmann.
\newblock Cluster formation in a stepping-stone model with continuous,
  hierarchically structured sites.
\newblock {\em Ann. Probab.}, 24(4):1926--1952, 1996.

\bibitem[EFS17]{EtheridgeFreemanStraulino2017}
Alison Etheridge, Nic Freeman, and Daniel Straulino.
\newblock The {B}rownian net and selection in the spatial
  {$\Lambda$}-{F}leming-{V}iot process.
\newblock {\em Electron. J. Probab.}, 22:Paper No. 39, 36, 2017.

\bibitem[EK86]{EK86}
Stewart~N. Ethier and Thomas~G. Kurtz.
\newblock {\em Markov processes}.
\newblock Wiley Series in Probability and Mathematical Statistics: Probability
  and Mathematical Statistics. John Wiley \& Sons Inc., New York, 1986.
\newblock Characterization and convergence.

\bibitem[EK19]{EtheridgeKurtz2019}
Alison~M. Etheridge and Thomas~G. Kurtz.
\newblock Genealogical constructions of population models.
\newblock {\em Ann. Probab.}, 47(4):1827--1910, 2019.

\bibitem[Eth00]{Etheridge2000}
Alison~M. Etheridge.
\newblock {\em An introduction to superprocesses}, volume~20 of {\em University
  Lecture Series}.
\newblock American Mathematical Society, Providence, RI, 2000.

\bibitem[Eva97]{Evans1997}
Steven~N. Evans.
\newblock Coalescing {M}arkov labelled partitions and a continuous sites
  genetics model with infinitely many types.
\newblock {\em Ann. Inst. H. Poincar\'e Probab. Statist.}, 33(3):339--358,
  1997.

\bibitem[EVY20]{EtheridgeVeberYu2020}
Alison~M. Etheridge, Amandine V\'{e}ber, and Feng Yu.
\newblock Rescaling limits of the spatial lambda-{F}leming-{V}iot process with
  selection.
\newblock {\em Electron. J. Probab.}, 25:Paper No. 120, 89, 2020.

\bibitem[GdHKK14]{GHKK14}
A.~Greven, F.~den Hollander, S.~Kliem, and A.~Klimovsky.
\newblock Renormalisation of hierarchically interacting {C}annings processes.
\newblock {\em ALEA Lat. Am. J. Probab. Math. Stat.}, 11(1):43--140, 2014.

\bibitem[GKW23]{GKW23}
Andreas Greven, Anton Klimovsky, and Anita Winter.
\newblock Ancestral web and measure-valued path representations.
\newblock in preparation, 2023.

\bibitem[GPW09]{GPW09}
Andreas Greven, Peter Pfaffelhuber, and Anita Winter.
\newblock Convergence in distribution of random metric measure spaces
  ({$\Lambda$}-coalescent measure trees).
\newblock {\em Probab. Theory Related Fields}, 145(1-2):285--322, 2009.

\bibitem[GPW13]{GPWmp13}
Andreas Greven, Peter Pfaffelhuber, and Anita Winter.
\newblock Tree-valued resampling dynamics {M}artingale problems and
  applications.
\newblock {\em Probab. Theory Related Fields}, 155(3-4):789--838, 2013.

\bibitem[GSW16]{GSW}
Andreas Greven, Rongfeng Sun, and Anita Winter.
\newblock Continuum space limit of the genealogies of interacting
  {F}leming-{V}iot processes on {$\mathbb Z$}.
\newblock {\em Electron. J. Probab.}, 21:Paper No. 58, 64, 2016.

\bibitem[JK14]{JansenKurt2014}
Sabine Jansen and Noemi Kurt.
\newblock On the notion(s) of duality for {M}arkov processes.
\newblock {\em Probab. Surv.}, 11:59--120, 2014.

\bibitem[Kal21]{Kallenberg2021}
Olav Kallenberg.
\newblock {\em Foundations of modern probability}, volume~99 of {\em
  Probability Theory and Stochastic Modelling}.
\newblock Springer, Cham, 2021.
\newblock Third edition [of 1464694].

\bibitem[KS88]{KonnoShiga1988}
N.~Konno and T.~Shiga.
\newblock Stochastic partial differential equations for some measure-valued
  diffusions.
\newblock {\em Probab. Theory Related Fields}, 79(2):201--225, 1988.

\bibitem[KS91]{KaratzasShreve1991}
I.~Karatzas and S.~E. Shreve.
\newblock {\em {B}rownian motion and stochastic calculus}.
\newblock Springer-Verlag, New York, 1991.

\bibitem[KS01]{KSt01}
Thomas~G. Kurtz and Richard~H. Stockbridge.
\newblock Stationary solutions and forward equations for controlled and
  singular martingale problems.
\newblock {\em Electron. J. Probab.}, 6:no. 17, 52, 2001.

\bibitem[LeC57]{LC57}
Lucien LeCam.
\newblock Convergence in distribution of stochastic processes.
\newblock {\em Univ. Calif. Publ. Statist.}, 2:207--236, 1957.

\bibitem[Lig85]{Liggett:85}
Thomas~M. Liggett.
\newblock {\em Interacting particle systems}, volume 276 of {\em Grundlehren
  der Mathematischen Wissenschaften [Fundamental Principles of Mathematical
  Sciences]}.
\newblock Springer-Verlag, New York, 1985.

\bibitem[LR16]{LoehrRippl2016}
Wolfgang L\"ohr and Thomas Rippl.
\newblock Boundedly finite measures: separation and convergence by an algebra
  of functions.
\newblock {\em Electron. Commun. Probab.}, 21:Paper No. 60, 16, 2016.

\bibitem[LV22]{louvet2022growth}
Apolline Louvet and Amandine V{\'e}ber.
\newblock Growth properties of the infinite-parent spatial lambda-fleming viot
  process.
\newblock {\em arXiv preprint arXiv:2205.03937}, 2022.

\bibitem[MT95]{MullerTribe1995}
C.~M\"{u}ller and R.~Tribe.
\newblock Stochastic p.d.e.'s arising from the long range contact and long
  range voter processes.
\newblock {\em Probab. Theory Related Fields}, 102(4):519--545, 1995.

\bibitem[Shi94]{Shiga1994}
Tokuzo Shiga.
\newblock Two contrasting properties of solutions for one-dimensional
  stochastic partial differential equations.
\newblock volume~46, pages 415--437. 1994.

\bibitem[SSV12]{SchillingSongVondracek2012}
Ren\'e~L. Schilling, Renming Song, and Zoran Vondra{\v c}ek.
\newblock {\em Bernstein functions}, volume~37 of {\em De Gruyter Studies in
  Mathematics}.
\newblock Walter de Gruyter \& Co., Berlin, second edition, 2012.
\newblock Theory and applications.

\bibitem[VW15]{VeberWakolbinger2015}
A.~V\'{e}ber and A.~Wakolbinger.
\newblock The spatial {L}ambda-{F}leming-{V}iot process: an event-based
  construction and a lookdown representation.
\newblock {\em Ann. Inst. Henri Poincar\'{e} Probab. Stat.}, 51(2):570--598,
  2015.

\end{thebibliography}

\end{document}